\documentclass{article}
\usepackage{titlesec}
\usepackage{graphicx}%
\usepackage{multirow}%
\usepackage{amsmath,amssymb,amsthm,amsfonts,mathrsfs,hyperref}%
\usepackage{tabularx}
\usepackage{colortbl} 
\usepackage{comment} 
\usepackage[title]{appendix}%
\usepackage[table]{xcolor}
\usepackage{textcomp}%
\usepackage{manyfoot}%
\usepackage{booktabs}%
\usepackage{algorithmicx}%
\usepackage{listings}%
\usepackage{longtable}
\usepackage[T1,T2A]{fontenc}
\usepackage{pgfplots,subcaption}
\usepackage{array}
\usepackage[margin=1in]{geometry}
\usepackage{dsfont, cancel}
\usepackage[normalem]{ulem}
\usepackage{microtype}

\pgfplotsset{compat=1.18} 

\newtheorem{theorem}{Theorem}[section]
\newtheorem{corollary}[theorem]{Corollary}
\newtheorem{lemma}[theorem]{Lemma}

\newtheorem{proposition}[theorem]{Proposition}

\newtheorem{conjecture}[theorem]{Conjecture}
\newtheorem{definition}[theorem]{Definition}

\newtheorem{remark}[theorem]{Remark}

\DeclareUnicodeCharacter{00A0}{} 

\title{T\'oth's buses and the ``detachment process''}
\author{J\'anos Engl\"ander\thanks{Department of Mathematics, University of Colorado Boulder, Boulder, CO-80309, USA,\ \href{mailto:englandj@colorado.edu}{englandj@colorado.edu}}}

\date{\today}

\begin{document}
\maketitle

\begin{abstract}
This paper introduces the \textbf{detachment process}, a novel, time-inhomogeneous Markov process inspired by I. P. T\'oth's problem \cite{Toth} concerning the number of ``lonely passengers'' (those without companions) when $n$ passengers are seated independently and uniformly in $k$ initially empty buses. T\'oth showed that this number is stochastically non-decreasing in $k$ for fixed $n$ (see also Haslegrave's work \cite{Haslegrave}).

\smallskip
\noindent\underline{\textsf{The new model as dynamical version; main concepts:}}
We extend T\'oth's model by treating the number of buses $k$ as a time parameter. Specifically, for a fixed number of passengers $n$, the state of our Markov process at time $k \ge 1$ is exactly T\'oth's configuration $(n, k)$. (We formally extend the process definition for all $t \in [1, \infty)$.) These processes can be coupled for all $n \ge 1$, and this larger coupled process is what we dub the \textbf{detachment process}.

Our investigation focuses on properties related to detachment, clumping, the number of lonely passengers and of non-empty  buses. The central notion is \textbf{detachment}, which occurs at time $k$ if  every passenger occupies a distinct bus; we say the process is \textbf{in a state of detachment} at $k$. A \textbf{detachment time} $k$ is when the process transitions from a non-detached state at $k-1$ to a detached state at $k$.

\smallskip
\noindent\underline{\textsf{Key time scales:}}
We identify \textbf{four critical time scales}---linear, quadratic, and log-corrected linear or quadratic in the number of passengers, $n$---that govern the process's properties.
\begin{itemize}
\item 
For $c>0,n\to\infty$, one has that  $\frac{1}{n}\log P(\text{detachment before}\ cn)\in (C_1(c),C_2(c)),$ while the process exhibits ``almost complete detachment'' at any super-linearly large time $k$ (in $n$). However,  full detachment (the state of detachment) first occurs significantly later.
\item   For $n \gg 1,y>0$, define the time scale:
$$k(n,y) := \frac{n^{2}}{2\,(2\log n - 2\log\log n + y)} = \frac{n^{2}}{4\log n}\left\{1+\frac{\log\log n - y/2}{\log n} + \mathcal{O}\!\left(\left[\frac{\log\log n}{\log n}\right]^2\right)\right\}$$
If $e(n,k)$ is the expected number of times the process is in a state of detachment up to time $k$, then $\lim_{n\to\infty} e(n,k(n,y)) = c(y) := e^{-y}/8.$ 
\item The time of first  detachment, $\hat\tau^{(n)}$ happens sometime between the log-corrected and pure quadratic (in $n$) times. The time of permanent (eventual) detachment, $\tau^{(n)}$, scales \emph{exactly} quadratically, and the law of its scaling limit is the inverse exponential distribution. In fact ``zero percent'' detachment occurs until $o(n^2)$.
\item 
The relevance of the purely quadratic scaling is further underscored by introducing the $\{0,1\}$-valued  process $X^{(n)}$ on $(0,\infty)$, defined by $\;X^{(n)}_t := \mathds{1}(\text{state of detachment at time } t n^2)\;$. We show that, as $n \to \infty$, the processes $X^{(n)}$ converge in the sense of finite-dimensional distributions to a limiting c\`adl\`ag process, although tightness fails.

\item The log-corrected \emph{linear} time scale is shown to be critical for the concentration of the number of lonely passengers around its mean, revealing a dependence on a fine-tuning constant. We also discuss a related Poisson approximation.

\end{itemize}

\noindent\underline{\textsf{Further findings:}}
We investigate (relative) clumping. We also explore why modeling the number of passengers with a Poisson distribution simplifies the analysis of T\'oth's original model. To aid this derivation, we introduce a comparison theorem for binomial distributions, originally obtained by J. Najnudel \cite{Najnudel}, along with a novel proof.
\end{abstract}
\subsection*{Keywords}  Detachment process, T\'oth's bus problem, lonely passengers problem, occupancy problem, detachment, almost detachment, stochastic dominance, inhomogeneous Markov chain, Stirling numbers, inverse exponential distribution,
clumping, scaling limits, Maxwell-Boltzmann statistics, Poisson point process, Poissonization, Poisson approximation, Chen-Stein method, hypergeometric series/function.
\subsection*{Subject Classification} 
\noindent Primary: 60J10; Secondary: 60C05
\tableofcontents
\section{Introduction, and summary of results}\label{sec: intro}
Recently I. P. T\'oth \cite{Toth} stated\footnote{The problem was originally inspired by a question 
concerning graphs, posed by his brother L. M. T\'oth.} and proved the following very interesting claim. Let $n\ge 1$ passengers sit uniformly randomly and independently from each other in $k\ge 1$ buses, which were originally empty. Consider the number of lonely passengers (without having any fellow travelers on their buses). Then for $n$ fixed, this number is stochastically  non-decreasing in $k$. In particular, the probability that there exists at least one lonely passenger is  non-decreasing in $k$. (T\'oth proved that it is in fact strictly increasing.)

This statement is particularly interesting because, although the question is elementary (especially when one asks only about the probabilities), the problem remained open for several years, until the surprisingly complicated but quite ingenious proof by T\'oth himself surfaced \cite{Toth}, using Markovian tools and multiple clever coupling arguments. 

The reader is warned that it is very easy to make a subtle mistake and end up with a proof that looks simple and right but actually has a flaw. Quite a few such attempted proofs were seen (and produced) by the author of this paper too. 

Finally, a much shorter, new proof has very recently been offered by Haslegrave in a preprint \cite{Haslegrave}.

Our goal here is to popularize T\'oth's ``lonely passengers'' and show how a dynamical version offers many additional interesting problems.

\subsection{The \texorpdfstring{$n$}{n}-detachment process}
In this note we revisit I. P. T\'oth's model \cite{Toth} and consider the number of buses as the time parameter; hence, for each $n\ge 1$ we create a process for which its state at time $k\ge 1$ is exactly T\'oth's model with parameters $n,k$.

\begin{definition}[The $n$-detachment process]
We start with $X_1$ which is the T\'oth model with n passengers and $k=1$ bus.
At time $k\ge 2$ each passenger (sitting currently in one of $k-1$ buses), independently, has the $1/k$ chance to relocate to the
newly available $k$th bus. A bit more formally, we can say that the process is a measure-valued ``occupancy process'' on $\mathbb Z_+=\{1,2,...\}$ and the mass on $z\in \mathbb Z_+$ is the number of passengers in bus number $z$.
We call $X_1,X_2,...$ the {\bf $n-$detachment process}.
\end{definition}
\noindent{\bf Convention:} We extend the definition of the $n$-detachment process from integer times to positive real times so that at time $t>1$ the state of the process is the same as at time $\lfloor t\rfloor$. This way we avoid the frequent use of the integer value sign. 
\begin{remark}\label{rem: onedimmarg}
By a trivial inductive argument, the state
$X_k$ is equivalent to the T\'oth model with $(n,k)$. That is, the one dimensional marginals of the $n$-detachment process are the T\'oth models for the given $n$ and $k=1,2,...$. (In terms of ``occupancy problems,'' we have the Maxwell-Boltzmann statistics with $n$ balls and $k$ boxes.)
$\hfill\diamond$\end{remark}
We continue with  definitions.
\begin{definition} A passenger is called {\bf lonely} if there are no other passengers in the bus where the passenger is sitting. Of course, being lonely depends on time.
Let $L_k$ denote the number of lonely passengers at time $k$; let $N_k$ denote the number of nonempty buses in $X_k$, or the ``support size'' of $X_k$. 
\end{definition}
\begin{definition}[Time of first and of permanent detachment]
Let $\tau^{(n)}$ denote the ``start of permanent detachment,'' ($=$ last detachment time) that is,
$$\tau^{(n)}:=\min\{k\ge 1\mid L_{k+j}=n,\ \forall j\ge 0\}.$$
The {\it first} detachment time is defined as $$\hat\tau^{(n)}:=\min\{k\ge 1\mid L_{k}=n\}.$$
\end{definition}
Although $\tau^{(n)}$ is a priori $[0,\infty]$-valued, later we will prove that it is a.s. finite-valued (hence $\hat\tau^{(n)}$ too).
Note that $\tau^{(n)}=\min\{k\ge 1\mid N_{k+j}=n,\ \forall j\ge 0\}$ holds too.
\begin{definition}[State of detachment vs. detachment time]
When $L_k=n$, we say the process is in a \emph{state of detachment}. But $k$ is called a \emph{detachment time} only when $L_{k-1}<n$ and $L_k=n.$

So, for example, if
$k>\tau^{(n)}$ then $k$ is not a detachment time, but the process is in a state of detachment at $k$.
\end{definition}

\subsection{Coupling the processes for all parameters: the detachment process}\label{subsect: couple}
One does not have to define the detachment process for each $n$ separately, as they can be coupled in one ``big'' process. That process starts with infinitely many passengers, labeled by positive integers. The reader should recall that each passenger chooses the newly available bus independently, and this will also be true when we start with infinitely many passengers. This latter process will be called the {\bf detachment process}.
    
It is clear that the sub-process formed by following the first $n$ passengers only will then be the $n$-detachment process. In this case the random times $\tau^{(n)}$ and $\hat\tau^{(n)}$ are coupled in a monotone way, that is,
    $\tau^{(1)}\le \tau^{(2)}\le ...$ and
    $\hat\tau^{(1)}\le \hat\tau^{(2)}\le ...$
\subsection{Comparison with classical occupancy problems}

When one studies the law of $N_k$, the number of nonempty buses for a fixed $k$, 
the model reduces to the classical occupancy problem. The behavior of $N_k$ as both 
$n$ and $k$ tend to infinity is well understood in a variety of scaling regimes; 
see the fundamental monograph \cite{Allocationbook}. That work also establishes 
several limit theorems for the joint distribution of 
\[
(Y_0, Y_1, Y_2, \dots, Y_n),
\]
where $Y_r$ denotes the number of buses containing exactly $r$ passengers, under 
a number of asymptotic assumptions. An interesting feature is that the treatment of $Y_1$---the number of ``lonely'' 
passengers---requires techniques different from those used for $Y_r$ with 
$r \neq 1$. These issues, however, fall outside the scope of the present article, 
and we will not pursue them here.

Our approach is different.  We are interested in \emph{process properties} and not solely one-dimensional marginal distributions. For instance, we will investigate the behavior of the times $k$ when $L_k=n$ occurs (the process is in a ``state of detachment''), and the behavior of $\tau^{(n)}$ and $\hat\tau^{(n)}$, which cannot be determined by results on occupancy problems.

\subsection{Notation}

{\bf General notation:} Let $n\ge 1$ be given. As mentioned above, time is indexed by $k$ to be consistent with the notation in \cite{Toth}; later when we also use non-integer times, we use $t$. The notation $a_N\sim b_N$ will mean that $\lim_{N\to\infty}a_N/b_N=1;$ the notation $a_N\asymp b_N$ will mean that there exist $0<c<C$ such that $c<a_N/b_N<C$ for all large $N$.  As usual, $f(n)\ll g(n)$ will denote that $f(n)=o(g(n))$ as $n\to\infty$ and $f(n)=\mathcal{O}(g(n))$ will denote that $f(n)\le Cg(n)$ with some $C>0.$ When in a sum $\sum_{j=a}^b$ the upper limit is not an integer, the understanding will be that the upper limit is in fact $\lfloor b\rfloor$.

Furthermore, $\mathsf{sgn}$ will denote the usual sign function,
and we will reserve $\log$ to denote natural logarithm.
Finally, for a process $\xi$ we will write $\xi_i$ as well as $\xi(i)$ to denote its value at $i$, whichever is more convenient, and we will be equally flexible with numerical sequences: $a_n=a(n)$.

\medskip
{\bf Rising and falling factorials, hypergeometric series:}
In this paper $a^{\overline{r}}:=a(a+1)\cdots(a+r-1)$ will denote the \emph{rising factorial}
(rising Pochhammer symbol),  and $a^{\underline{r}}:=a(a-1)\cdots(a-r+1)$ the \emph{falling factorial},
related by $(-a)^{\overline{r}}=(-1)^r\,a^{\underline{r}}$, with the convention that $(\cdot)^{\overline 0}=(\cdot)^{\underline 0}=1$. In particular, $a^{\underline{r}}=0$ if $r>a>0.$
The falling factorial will also be denoted sometimes by $(a)_n$. (\underline{Warning:} in some texts $(a)_n$  denotes the rising one.) 
The \emph{generalized hypergeometric series} ${}_2F_0$ is defined, for complex parameters
$a_1,a_2$ and argument $z$, as
\[
{}_2F_0(a_1,a_2;\,-;\,z)
=\sum_{r=0}^{\infty}\frac{a_1^{\overline{r}}\,a_2^{\overline{r}}}{r!}\,z^r.
\]

We will be particularly interested in ${}_2F_{0}(-a,-b;\,-;\,z)$ when $a,b>0$. There are equivalent versions with rising and falling factorials:
\[
{}_2F_0(-a,-b;\,-;\,z)
=\sum_{r=0}^{\infty}\frac{(-a)^{\overline{r}}\;(-b)^{\overline{r}}}{r!}\,z^{r}
=\sum_{r=0}^{\infty}\frac{a_{\underline{r}}\,b_{\underline{r}}}{r!}\,z^{r}.
\ 
\]
In the binomial-coefficient form, one uses the identity
\[
x_{\underline{r}}=r!\binom{x}{r},\qquad
\binom{x}{r}:=\frac{\Gamma(x+1)}{\Gamma(r+1)\Gamma(x-r+1)},
\]
and obtains that
\[
{}_2F_0(-a,-b;\,-;\,z)
=\sum_{r=0}^\infty r!\,\binom{a}{r}\,\binom{b}{r}\;z^r.
\;
\]

\smallskip
If either \(a\) or \(b\) is a positive integer, the falling factorials
terminate and the series becomes a finite polynomial; otherwise,
\({}_2F_0\) is understood as a formal  series.

\subsection{Outline of the paper and summary of results}
In Section \ref{sec: basic} we review the results  by T\'oth and Haslegrave. We then obtain some basic results on detachment and on transition probabilities.
As far as the number of lonely passengers at time $k$ ($=L^n_k$) is concerned, we prove that
the critical regime for the concentration of $L_{k}^n$ around its mean is $k(n)=\dfrac{n}{\alpha \log n}$, with a phase transition  at $\alpha=1$.
We also show that if $c>1$ and $k(n)\le \dfrac{n}{c\log n}$ then for any $m\ge 1$, $$ P^{(n)}(\exists \text{ at least m lonely passengers})= P^{(n)}(L^n_k\ge m)\le 
\frac{1}{m\cdot n^{c-1+o(1)}},$$ as $n\to\infty$, 
and discuss a Poisson approximation for $L_{k}^n$ in the regime $k(n)=\dfrac{n}{\log(cn)}$.

Concerning the minimum/maximum and the range  of the process, we 
prove in Subsection \ref{subsec:minmax} that for $n\ge 1$ fixed, $\lim_k \mathsf{Law}(m_k/k)=\mathsf{Beta}(1,n);\ \lim_k \mathsf{Law}(M_k/k)=\mathsf{Beta}(n,1).$

Turning to Section \ref{sect:lon.det}, for the distribution of the last detachment time $\tau^{(n)}$, it is shown in Subsection \ref{subs: permanent} that $\tau^{(n)}$ 
    \begin{enumerate}
    \item  is almost surely finite, with a tail asymptotical with $2n(n-1)/k$ as $k\to\infty$ and thus has infinite expectation;
    \item  follows the law with distribution function
    \begin{align*}
   F^{(n)}(k)=P(\tau^{(n)}\le k)=
    \begin{cases}
 \frac{ \binom{k}{n} }{\binom{k+n-1}{n}},\ k=n,n+1,...\\
0,\ k<n.\\
\end{cases}
\end{align*}
\end{enumerate}
We then prove that the law of 
$\tau^{(n)}/n^2$ converges to IE(1), which is the law of the inverse exponential distribution with unit parameter. 

With respect to the scaling for the \emph{first} detachment time, in Subsection \ref{subs: first.det} the following is demonstrated.  Let $(k_n)$ be a given sequence of positive integers.
\begin{enumerate}
    \item[(i)] Let $n^2=\mathcal{O}(k_n)$. Then $\liminf_n P(\hat \tau^{(n)}\le k_n)\ge \liminf_n P(L_{k_{n}}=n)>0.$ In particular, $\hat \tau^{(n)}/n^2$ does not converge to infinity in law.
    \item[(ii)] Let $k_n\le \frac{n^2}{c\log n},\ c>4$, or more generally, assume that $(n-1)^2/(2k_n)-\log k_n\to\infty$. Then
    $\lim_{n\to\infty} P(\hat \tau^{(n)}\le k_n)=0,$ and so for $k_n=o( \frac{n^2}{\log n}),$ one has $\hat \tau^{(n)}/k_n\to\infty$ in law.
    \item[(iii)] (large deviations) For $c>1$, 
    $$ -\frac{1}{2c}-\,\frac{1}{6c^2}\le\liminf_{n\to\infty}\frac{1}{n}\log P(\hat\tau^{(n)}\le cn)
\le\limsup_{n\to\infty}\frac{1}{n}\log P(\hat\tau^{(n)}\le cn)\le -\frac{1}{2c}.$$
    \end{enumerate}
    This is then complemented by the result that there is ``zero percent detachment until $o(n^2)$,'' that is, if $k=k(n)=o(n^2)$ as $n\to\infty$ then one has $\lim_{n\to\infty}\mathsf{R}^n_{k(n)}=0$ in probability, where $\mathsf{R}^n_k$ denotes the percentage of times (up to $k$) when the process with $n$ passengers is in a state of detachment. For $k_n\gg n^2$, however, the limit is one.

Section \ref{sec:scalinglimit} provides a scaling limit result for the $\{0,1\}$-valued  process $X^{(n)}$ on $(0,\infty)$, defined by $\;X^{(n)}_t := \mathds{1}(\text{state of detachment at time } t n^2)\;$, in the sense of finite-dimensional distributions.

Section \ref{sec: further} treats further scaling results. Focusing then on the \emph{expected number of detachment states}, it is shown  that the critical order is once again $k(n)=\frac{n^{2}}{4\log n}$. More precisely, if 
$k(n)\sim\dfrac{n^{2}}{c\log n},\ n\gg 1,$ and $c>4$
then $e(n,k)\to 0$ as $n\to\infty$, i.e., $\sum_{j=1}^{k(n)}X_j\to 0$ in $L^1$, while 
if $c<4$ then $e(n,k)\to \infty.$ 
A fine-tuning result then says that if $y>0$ and 
\begin{align}\label{eq: define.k(n)}
k(n)=k(n,y):=\frac{n^{2}}{2\,(2\log n-2\log\log n+y)}\sim \frac{n^{2}}{4\log n},
\end{align}
then 
$
\lim_{n\to\infty} e(n,k(n)) = \dfrac{e^{-y}}{8}=:c(y).
$
Informally, with $0<\mathrm{const}_1<\mathrm{const}_2$, the critical time window is  about $$\left[\frac{n^{2}}{4\log n}\left\{1+   \frac{\log\log n+\mathrm{const}_1}{\log n}\right\},\frac{n^{2}}{4\log n}\left\{1+\frac{\log\log n+\mathrm{const}_2}{\log n}\right\}\right]$$ for large $n$, and the expected number of times in a state of detachment  there is about \\
$\frac{1}{8}\left(\exp(-2\,\mathrm{const}_1)-\exp(-2\,\mathrm{const}_2)\right)$.

As far as ``almost detachment'' is concerned, we demonstrate that \emph{any} super-linearly large time is sufficient for that to occur. Namely, we show that
assuming  $n\ll k=k(n),$ as $n\to\infty,$ the limits
$\dfrac{N_{k(n)}}{n}\to 1$ and $\dfrac{L_{k(n)}}{n}\to 1$ hold in probability.

Next, in Section \ref{sec: support size}, we turn to the analysis of the support size. We show the following.
If $n\ge 2$ and $S(a,b)$ denotes the \emph{Stirling number of the second kind} with parameters $a,b$ then for $m\le \min\{n,k\}$,
   $$P^{(n)}(N_k\ge m)=\frac{k!}{(k-m)!}\left(k^{-m}+a(m,n,k)\right),$$
   where
$m^{-n}S(n,m)(m/k)^m<a(m,n,k)<m^{-n}S(n,m)(m/k)^n.$ In addition, concerning the  
distribution of the support size 
for $n$ passengers, the generating function of $N_k$ is shown to be
\[
G_n(z)=\sum_{j=1}^{k}\binom{k-1}{j-1}\left(\frac{j}{k}\right)^{n-1} 
z^{\,j}(1-z)^{\,k-j}.
\]

In Section \ref{sec: clumping} we introduce measurements of ``clumping'' and ``relative clumping'' and prove that while complete “decumpling” (i.e., complete detachment) requires a roughly quadratic time interval, relative clumping significantly decreases from the initial value ($=\log n$) below a constant (dependent on $c$) already within $cn$ time.

In Section \ref{sec: Poissonization}, we  explore the reasons why incorporating a Poisson distribution for the number of passengers simplifies the analysis of T\'oth's model. To aid this derivation, we introduce a comparison theorem for binomial distributions, originally obtained by J. Najnudel, along with a novel proof.

Section \ref{sec: open pr}) proposes some open problems, and the paper concludes with two appendices.

\subsection{Table with four time scales} The following table outlines the paper's key findings, emphasizing the distinct time-scales and behaviors.

\medskip 

\setlength{\tabcolsep}{10pt}     
\renewcommand{\arraystretch}{1.15} 
\rowcolors{2}{gray!20}{white}

\noindent
\begin{tabularx}{\textwidth}{|*{5}{>{\centering\arraybackslash}X|}}
\hline
\rowcolor{blue!25}
\textbf{}
& \textbf{linear scale with $c\log$-correction}
& \textbf{super-linear scale, but $\ll$ next column}
& \textbf{quadratic scale with $c\log$-correction}
& \textbf{purely quadratic scale} \\
\hline
\textbf{full detachment}
& no & no & no for large $c$ & yes \\
\hline
\textbf{almost detachment}
& no & yes & yes & yes \\
\hline
\textbf{expected times in state of detachment}
& tends to $0$
& tends to $0$
& depends on fine tuning constant
& tends to $\infty$ \\
\hline
\textbf{scaling limit for fidis}
& no & no & no & yes \\
\hline
\end{tabularx}

\section{Some known results and basic calculations}\label{sec: basic}
\subsection{T\'oth's results}
Recall that $L_k$ denotes the number of lonely passengers at time $k$ and $N_k$ is the number of nonempty buses.
\begin{theorem}[T\'oth's bus theorem]\label{thm: Toth}
$L$ is stochastically non-decreasing in $k$, that is $L_k$ is stochastically dominated by $L_{k+1}$. In fact $P(L_k\ge 1)$ is strictly increasing in $k$.
\end{theorem}
\begin{proof}
It follows from T\'oth's theorem in \cite{Toth} and Remark \ref{rem: onedimmarg}.
\end{proof}
The following result is stronger than Theorem \ref{thm: Toth}.

\begin{theorem} [T\'oth]\label{thm: N.by.Toth} For the process $N$ we have that
\begin{enumerate}
\item $N$ is stochastically non-decreasing in $k$, that is $N_k$ is stochastically dominated by $N_{k+1}$;
\item for any $k,k'>0,$ the conditional variable
$L_k\mid \{N_k=a\}$ is stochastically dominated by
$L_k'\mid \{N_k'=b\}$, when $a\le b$ and $a\le k\wedge n$ and $b\le k'\wedge n.$
\end{enumerate}
\end{theorem}
\begin{proof}
These follow from results in T\'oth's paper \cite{Toth}.
\end{proof}

\subsection{Single time detachment}\label{rem: two calcul} Let $n\le k$.
Denoting $(a)_r:=a^{\underline{r}}$, the probability that at $k$
the process is in a \emph{state of detachment} is clearly
\begin{align}\label{eq:pi.clear}
\pi_{n,k}:=P(L_k=n)=\frac{(k)_n}{k^{n}},
\end{align}
while the probability that $k$ is a \emph{detachment time} is 
\begin{align}\label{eq: pr.det.time}
p_{n,k}:=\frac{n(n-1)(k-2)_{n-2}}{k^{n}}.
\end{align}
To see why the latter identity holds, write $$p_{n,k}=P(L_{k-1}<n,L_k=n)=\frac{(k)_n}{k^{n}}\cdot P(L_{k-1}<n\mid L_k=n).$$
Since in backward time the configuration at $k-1$ is obtained from that at time $k$ by sending all passengers in bus label $k$ to the buses with labels less than $k$, but under the conditioning, in the $k$th bus there is at most one passenger, it follows that
\begin{align}\label{eq:pnk}
p_{n,k}&=\frac{(k)_n}{k^{n}}\cdot P(L_{k-1}<n,\ k^{th}\text{bus is non-empty at time }k\mid L_k=n)=\frac{(k)_n}{k^{n}}\cdot \frac{n(k-1)_{n-1}}{(k)_n}\cdot \frac{n-1}{k-1}\nonumber\\
&= \frac{n(k-1)_{n-1}}{k^{n}}\cdot \frac{n-1}{k-1}=\frac{n(n-1)(k-2)_{n-2}}{k^{n}},
\end{align}
as claimed.

\medskip 
 Bounding $\pi_{n,k}$ is standard (``birthday problem approximation''). Recalling that for $k\ge n,$
\[
\pi_{n,k}=\frac{(k)_n}{k^n}
= \prod_{j=1}^{n-1}\left(1 - \frac{j}{k}\right);\ i.e.\ 
\log  \left(\frac{(k)_n}{k^n}\right)
= \sum_{j=1}^{n-1} \log \!\left(1 - \frac{j}{k}\right),
\]
and expanding $\log (1-x) = -x - \tfrac{x^2}{2} - \cdots$ gives the bound
\[
\log  \left(\frac{(k)_n}{k^n}\right)
< -\frac{1}{k}\sum_{j=1}^{n-1} j
-\frac{1}{2k^2}\sum_{j=1}^{n-1} j^2.
\]
Thus, we have
\[
\log  \left(\frac{(k)_n}{k^n}\right)
< -\frac{n(n-1)}{2k}
- \frac{n(n-1)(2n-1)}{12k^2}.
\] 
A similar computation (using that for $0<x<1,$
$
\sum_{m=3}^{\infty} \frac{x^{m}}{m}
\;\le\;
\sum_{m=3}^{\infty} \frac{x^{m}}{3}
\;=\;
\frac{x^{3}}{3(1-x)}
$) gives the lower bound, and altogether we have
\begin{align}\label{eq: birthday}
-\frac{n(n-1)}{2k}
-\frac{n(n-1)(2n-1)}{12k^2}
-\frac{n^2(n-1)^2}{12\,k^2\,(k-n+1)}
\;\le\;
\log \pi_{n,k}
<
-\frac{n(n-1)}{2k}
-\frac{n(n-1)(2n-1)}{12k^2}.
\end{align}
Hence, for example, if $k=k(n)\ll n^2$ and $n\to\infty$ then $\pi_{n,k(n)}\to 0,$ while if $k(n)\sim cn^2$ then  
\begin{align}\label{eq: asympt.log}
-\log \pi_{n,k(n)}= \frac{1}{2c} +\mathcal{O}(1/n).
\end{align}
Note also that the product formula immediately shows  that $\pi_{n,k}$ is increasing in $k$.

Another consequence of \eqref{eq: birthday} is the following result on the proportion of times spent in states of detachment.
Let $D^n_k$ denote the number of times up to $k$ when the $n$-detachment process is in a detached state:
$$D^n_k:=|\{i\in \{1,2,...,k\}:\ L_i=n\}|,$$
and let $\mathsf{R}^n_k$ be the proportion of those times among all times $1,2,...,k$, that is, $\mathsf{R}^n_k:=D^n_k/k$. 
\begin {corollary}[Zero percent detachment until $o(n^2)$] 
If $k=k(n)=o(n^2)$ as $n\to\infty$ then one has
$\lim_{n\to\infty}\mathsf{R}^n_{k(n)}=0$ in probability (even in $L^1$).
\end{corollary}
\noindent \underline{Note:} It follows from Theorem \ref{thm:IE} of the next section that 
$\lim_{n\to\infty}\mathsf{R}^n_{k(n)}=1$ in probability if $k(n)\gg n^2.$
\begin{proof}
 Let $A_k^n:=\{L_k=n$\}. (Clearly, for $k< n,$ the event $P^{(n)}(A)_k^n=0$.) Then
$$E\mathsf{R}^n_k=(1/k)ED_k^n=\frac{1}{k}\sum_{j=1}^k P(A^n_j)=
\frac{1}{k}\sum_{j=1}^k \pi_{n,j},$$  and since $\pi_{n,j}$ is increasing in $j$,
it is enough to show that 
$\lim_{n\to\infty}\pi_{n,k(n)}=0$. But this follows from the assumption, as noted after \eqref{eq: birthday}.
\end{proof}
Let us now turn the asymptotic behavior of  $p_{n,k(n)}$.
\begin{proposition}
Consider a  sequence $k_n\to\infty$ with $k_n\ge n$. 
\begin{itemize}
\item[\textup{(i)}] For \emph{any such sequence} , one has
$p_{n,k_n}=\mathcal{O}(n^{-2}).$
\item[\textup{(ii)}] If $k_n\sim c\,n^2$ with $c>0$,
then
\[
p_{n,k_n}
\sim \frac{1}{c^2}\,e^{-1/(2c)}\,\frac{1}{n^2};
\]
\item[\textup{(iii)}] If $k_n\gg n^2$,
then
$p_{n,k_n}\sim \dfrac{n^2}{k_n^2}.$
\end{itemize}
\end{proposition}
\begin{proof}
Write
\[p_{n,k}
= \frac{n(n-1)}{k^2}\prod_{j=2}^{n-1}\Bigl(1-\frac{j}{k}\Bigr),
\qquad k\ge n.\]
Then $$\log p_{n,k}=\log\left(\frac{n(n-1)}{k(k-1)}\right)+\log\pi_{n,k}.$$
By \eqref{eq: birthday} it follows that $\dfrac{n(n-1)}{k(k-1)}e^{I_1}\le p_{n,k}\le \dfrac{n(n-1)}{k(k-1)}e^{I_2}$,
where 
\begin{align}\label{eq: nice.bound}
I_1:=-\frac{n(n-1)}{2k}
-\frac{n(n-1)(2n-1)}{12k^2}
-\frac{n^2(n-1)^2}{12\,k^2\,(k-n+1)};\ 
I_2:=-\frac{n(n-1)}{2k}
-\frac{n(n-1)(2n-1)}{12k^2}.
\end{align}
For (i), let $\ell_n:=\frac{n-1}{2k}.$ Then $$p_{n,k}\le \dfrac{n(n-1)}{k(k-1)}e^{-\dfrac{n(n-1)}{2k}}=4(1+o(1))\ell_n^2e^{-n\ell_n}=
16e^{-2}n^{-2}(1+o(1)),$$
using that $\max_{x>0}x^2e^{-nx}=4/(e^{2}n^{2})$.

Finally, (ii) and (iii) immediately follow from \eqref{eq: nice.bound}.
\end{proof}

\subsection{Multi-step transitions and two-time detachment} 
It is important to observe that 
for the $n$-detachment process, the $l$-step transition ($l\ge 1$) from time $k\ge 1$ to time $k+l$
is equivalent to the following: Every passenger independently joins one of the newly available buses (labeled  $k+1,\dots, k+l$ ) with probability  $l/(l+k).$ 

To see why this equivalence is true, note that the configuration of passengers in the newly available buses is the same whatever the configuration on the first $k$ buses was at time $k$. Therefore conditioning on the state of the process at $k$ can be dropped when we \emph{only consider the newly available buses} at time $k+l$. But without conditioning, we know that at time $k+l$ the state of the process is precisely the T\'oth model with $n$ passengers and $k+l$ buses, the restriction of which on buses labeled  $k+1,\dots, k+l$ is distributed according to the previous paragraph.

A useful application of this fact is as follows. Let $\varpi(n,k,l):=
P^{(n)}(L_{k+l}= L_k=n)$ be the probability that the process is at a detached state at time $k$ as well as at time $k+l$, where $k,l\ge 1$ and $k\ge n$.
Using the above description of the multi-step transitions,  a straightforward computation (with the $\ell^r$ and also the $k^n$ terms canceling out) gives that for $0<k,l$,
\begin{align}\label{eq: Hypergeom}
\varpi(n,k,l)=\frac{(k)_n}{(k+l)^n}\sum_{r=0}^{n\wedge l}\binom{n}{r}\frac{(l)_{r}}{k^{r}}=\frac{(k)_n}{(k+l)^n}\, {}_2F_0(-n,-l;-;1/k),
\end{align}
where ${}_2F_0$ is the generalized hypergeometric function.
(In the sum we can just as well write $\infty$ as upper limit, as it terminates at $\min(n,l)$ anyway.) 
\subsection{Concentration for the number of lonely passengers, WLLN} 
Stressing the total number of passengers in the notation, let now $L^n_k=L^n_{k(n)}$ denote the count of lonely passengers for $n$ passengers and $k = k(n)$ buses, that is $L^n_k:=\sum_{i=1}^n \mathds{1}(\text{passenger i is lonely at time } k).$ In this subsection, as a warm-up, we check when $L^n_k$
is concentrated around its mean for $n\gg 1$ and for that purpose we define
\[
R_{n,k}:=\frac{\mathsf{Var}^{(n)}(L^n_k)}{(E^{(n)}[L^n_k])^2}.
\] We are interested whether
$\sigma(L^n_k)=o(E^{(n)}[L^n_k])\ (\text{i.e.}\ R_{n,k}\to 0)$ holds, in which case, by Chebyshev,
\begin{align}\label{eq: Chebyshev}
\lim_nP^{(n)}\left(\left|\frac{L^n_k}{E^{(n)}[L^n_k]}-1\right|>\varepsilon\right)=0
\end{align} for all $\varepsilon>0.$ Then $L^n_k$
is concentrated around $E^{(n)}[L^n_k]=n(1-1/k)^{n-1},$ and (since clearly $L^n_k\le n$), the limit in \eqref{eq: Chebyshev} implies in particular the Weak Law of Large Numbers: \begin{align}
\lim_nP^{(n)}\left(\left|\frac{L^n_k-E^{(n)}[L^n_k]}{n}\right|>\varepsilon\right)=0,\ \forall \varepsilon>0.
\end{align}
\begin{theorem}[Concentration around the mean; $n/\log n$ scaling]\label{thm: conc} Let $n\to\infty.$ Then

(i) the critical regime for $R_{n,k}\to 0$ is \(k(n)\sim \tfrac{n}{\alpha\log n}\), and then $R_{n,k}\to 0$ if and only if $\alpha<1$; if on the other hand $k(n)\gg \frac{n}{\log n}$ then $R_{n,k}\to 0$;

(ii) $c>1$ and $k(n)\le \dfrac{n}{c\log n}$ implies that for any $m\ge 1$, $$ P^{(n)}(\exists \text{ at least m lonely passengers})= P^{(n)}(L^n_k\ge m)\le 
\frac{1}{m\cdot n^{c-1+o(1)}},$$ as $n\to\infty$.
\end{theorem}

\begin{remark}[Poisson approximation]
As the proof reveals (see the asymptotics in \eqref{eq:elso}), if $k(n)=\frac{n}{\log(cn)},\ c>0$ (now constant is inside the logarithm!), then $M_{n,k}=k\cdot q_{n,k}\to \frac{1}{c}$, and this raises the hope that the law of $L_{n,k}$ can be approximated by the  Poisson distribution with parameter $1/c$, despite the lack of independence of the Bernoullis. This can indeed be done, as an application of the Chen-Stein method to the Birthday Problem with $n$ persons and k days --- see section 4.1 in \cite{ChenSteinPoisson} for the proof and error bounds.
\end{remark}

\begin{proof}(of Theorem \ref{thm: conc})
First, with $X_i:=\mathds 1\{\text{bus }i\text{ has exactly one passenger}\}\ i=1,\dots k,$ and $q_{n,k} :=P^{(n)}(X_1=1)= \dfrac{n}{k}\left(1-\dfrac{1}{k}\right)^{n-1}$, one has
\[
L^n_k=\sum_{i=1}^k X_i, 
\qquad 
M_{n,k}:=E^{(n)}[L^n_k]=k q_{n,k}
= n\left(1-\frac{1}{k}\right)^{n-1},
\]
while for $i\neq j$,
\[
A_{n,k}:=P^{(n)}(X_i=X_j=1)
= \frac{n(n-1)}{k^2}\left(1-\frac{2}{k}\right)^{n-2}.
\]
Hence
\begin{align*}
V_{n,k}:=\mathrm{Var}^{(n)}(L^n_k)&=k p_{n,k}(1-p_{n,k})+k(k-1)\!\left[P^{(n)}(X_i=X_j=1)-p_{n,k}^2\right]\\
&=
k \frac{n}{k}\left(1-\frac{1}{k}\right)^{n-1}\left(1-\frac{n}{k}\left(1-\frac{1}{k}\right)^{n-1}\right)\\
&\qquad +k(k-1)\!\left\{\frac{n(n-1)}{k^2}\left(1-\frac{2}{k}\right)^{n-2}-\left[\frac{n}{k}\left(1-\frac{1}{k}\right)^{n-1}\right]^2\right\},
\end{align*}
that is,
\begin{align}
V_{n,k}
&= k n\frac{1}{k}\Big(1-\frac{1}{k}\Big)^{n-1}
\Big(1-n\frac{1}{k}\Big(1-\frac{1}{k}\Big)^{n-1}\Big)
+k(k-1)\!\left[A_{n,k}-\frac{M_{n,k}^2}{k^2}\right]\nonumber\\
&= M_{n,k}\Big(1-\frac{M_{n,k}}{k}\Big) + k(k-1)A_{n,k} - M_{n,k}^2\cdot \frac{k-1}{k}
= M_{n,k} + k(k-1)A_{n,k} - M_{n,k}^2.
\end{align}
Therefore,
\[
R_{n,k}
=\frac{1}{M_{n,k}}+\frac{k(k-1)\,A_{n,k}}{M_{n,k}^2} -1.
\]
For the rest of the calculations, we will use that  $k(n)\to\infty$.
By the expansion $\log(1-x)=-\sum_{j\ge 1}x^j/j$, we have
\begin{align}\label{eq:elso}
\left(1-\frac{1}{k}\right)^{n-1}
= \exp\left(-\frac{n - 1}{k} -  \mathcal{O}\!\Big(\frac{n}{k^2}\Big)\right);\ \ M_{n,k}=n\exp\left(-\frac{n - 1}{k} -  \mathcal{O}\!\Big(\frac{n}{k^2}\Big)\right);
\end{align}
\[
\left(1-\frac{2}{k}\right)^{n-2}
= \exp\left(-\frac{2(n - 2)}{k} -  \mathcal{O}\!\Big(\frac{n}{k^2}\Big)\right);\quad A_{n,k}
= \frac{n(n-1)}{k^2} \exp\left(-\frac{2(n - 2)}{k} -  \mathcal{O}\!\Big(\frac{n}{k^2}\Big)\right).
\]
This yields
\begin{align*}
R_{n,k}
&=\frac1n\exp\left(\frac{n-1}{k}+\mathcal{O}\left(\frac{n}{k^2}\right)\right)+\frac{k(k-1)\frac{n(n-1)}{k^2} \exp\left(-\frac{2(n - 2)}{k} -  \mathcal{O}\!\Big(\frac{n}{k^2}\Big)\right)}{n^2\exp\left(-2\frac{n - 1}{k} -  \mathcal{O}\!\Big(\frac{n}{k^2}\Big)\right)} -1\\
&=\frac{1}{n}\exp\left(\frac{n-1}{k}+\mathcal{O}\left(\frac{n}{k^2}\right)\right) + \left(1-\frac{1}{k}\right) \left(1-\frac{1}{n}\right) \exp\left(\frac{2}{k} + \mathcal{O}\!\Big(\frac{n}{k^2}\Big)\right) -1,
\end{align*}

\noindent Assume now that $k(n)=\dfrac{n}{\alpha\log n}$ with $\alpha>0.$ Then
\begin{align*}
R_{n,k}&=\exp\left((\alpha-1)\log n-1/k+\mathcal{O}\left(\frac{\log^2 n}{n}\right)\right)\\ &\ \ \ +  \exp\left(\frac{2}{k}+\log \left(1-\frac{1}{k}\right)+\log \left(1-\frac{1}{n}\right) + \mathcal{O}\left(\frac{\log^2 n}{n}\right)\right) -1.
\end{align*}
It follows that if $\alpha<1$, then $R_{n,k}\to 0$, while if $1\le\alpha$, then $R_{n,k}\not\to 0.$
It is also clear that if $k(n)\gg\dfrac{n}{\log n}$ then $R_{n,k}\to 0$.

Finally (ii) follows immediately from the asymptotics of $M_{n,k}$ in \eqref{eq:elso}.
\end{proof}

\subsection{Minimum, maximum and range}\label{subsec:minmax}
Recall that $k$ denotes both the time and the number of the newest bus that joins at that time. Let $m_k$ and $M_k$ denote the minimum and maximum bus numbers, respectively, among all nonempty buses at time $k$. That is, $m_k$ and $M_k$ are the bus numbers corresponding to the ``leftmost'' and ``rightmost'' occupied positions, assuming the buses arrive sequentially, always  to the right. Let $R_k:=M_k-m_k$ denote the ``range'' of non-empty buses.

Next, as another warm-up, we give some results about their large-time behavior. Since the large time limit of the point process scaled by $k$ corresponds to sampling $n$ points from the uniform distribution on $[0,1]$, the following result is hardly surprising. (But we think it is useful to record it.)
\begin{theorem}[Beta limits] 
\begin{enumerate} Let $m_k,M_k$ and $R_k$ be as in the previous paragraph.
\item   For $n\ge 1$ fixed, $$\lim_k \mathsf{Law}(m_k/k)=\mathsf{Beta}(1,n);\ \lim_k \mathsf{Law}(M_k/k)=\mathsf{Beta}(n,1),$$that is,  
$$\lim_{k\to\infty} P\left(\frac{m_k}{k}\le x\right)=1-(1-x)^n;\quad  \lim_{k\to\infty} P\left(\frac{M_k}{k}\le x\right)=x^n$$ for $x\in [0,1].$

\item The expectations and standard deviations satisfy 
\begin{align*}
\lim_k E(m_k/k)&=\frac{1}{n+1};\ \  \lim_k E(M_k/k)=\frac{n}{n+1};\ \ \lim_k E(R_k/k)=\frac{n-1}{n+1};\\
\lim_k\sigma(M_k/k)&=\lim_k\sigma(m_k/k)=\frac{1}{n+1}\sqrt{\frac{n}{n+2}}.
\end{align*}
\end{enumerate}
\end{theorem}
\begin{proof} Since the detachment process at the fix time $k$ is exactly the T\'oth model with parameters $n,k$, trivial computation yields that for $a=1/k,2/k,...,1,$
$$P\left(\frac{m_k}{k}\le a\right)=1-\left(\frac{k-ka}{k}\right)^n=1-(1-a)^n;\quad  P\left(\frac{M_k}{k}\le a\right)=\left(\frac{ka}{k}\right)^n=a^n$$  Taking limits in $k,$ and approximating $x$ by fractions $i/k$, the first part follows.

The convergence of moments follows from the fact that $m_k/k,M_k/k\in [0,1];$  this yields  part 2.\end{proof}
\section{First and last detachment times --- asymptotic results}\label{sect:lon.det}
\subsection{Permanent detachment time}\label{subs: permanent}
We first consider the time of permanent detachment $\tau^{(n)}$, for which we have an explicit result.
\begin{theorem}[Distribution of $\tau^{(n)}$]\label{thm: distr} The random variable $\tau^{(n)}$ 
\begin{enumerate}
    \item  is almost surely finite, with a tail asymptotical with $n(n-1)/k$ as $k\to\infty$ and thus has infinite expectation;
    \item has distribution function 
\begin{align}
F^{(n)}(k)
= P(\tau^{(n)} \le k)
= \frac{\binom{k}{n}}{\binom{k+n-1}{n}},
\end{align}
using the convention $\binom{a}{b}:=0$ when $a<b.$
\end{enumerate}
\end{theorem}
\begin{proof}
It is easy to check that the second statement implies the first. Indeed, for fixed $n$ and $k\to\infty,$
$ n!\binom{k}{n}=k^n-\frac12(n-1)nk^{n-1}+o(k^{n-1})$ while
$n!\binom{k+n-1}{n}=k^n+\frac12(n-1)nk^{n-1}+o(k^{n-1})$, and so
\begin{align*}
   P(\tau^{(n)}> k)=1-\frac{\binom{k}{n}}{\binom{k+n-1}{n}}
    &=1- \frac{k^n-\frac12 (n-1)nk^{n-1}+o(k^{n-1})}{k^n+\frac12(n-1)nk^{n-1}+o(k^{n-1})}=\frac{(n-1)nk^{n-1}+o(k^{n-1})}{k^n+(n-1)nk^{n-1}+o(k^{n-1})}\\
    &=\frac{(n-1)n+o(1)}{k+(n-1)n+o(1)}= \frac{n(n-1)}{k}+o\!\left(\frac{1}{k}\right).
\end{align*}
For the first statement, first recall that $\pi_{n,k}:=P(L_k=n)=\dfrac{(k)_n}{k^{n}}.$
(Clearly, $\pi_{n,k}=0$ when $k<n$.) Hence, for $k\ge n$,
\begin{align}
P(\tau^{(n)}\le k)&=\pi_{n,k}\prod_{i= k+1}^{\infty}\left(\left(\frac{i-1}{i}\right)^n+\frac{n}{i}\left(\frac{i-1}{i}\right)^{n-1}\right)\nonumber\\
&=
\pi_{n,k}\prod_{i= k+1}^{\infty}\left(\left(\frac{i-1}{i}\right)^{n}\frac{n+i-1}{i-1}\right),
\end{align}
    where, using independence, the product is the probability that after time $k$ the ``loneliness'' of the passengers is never violated, that is, it never occurs for any time $i>k$ that more than one of them jumps to the $i$th bus. 
    Considering the product only up to $k+N+1$, after a bit of algebra, one obtains that
    \begin{align}
      \lim_{N\to\infty}\prod_{i= k+1}^{k+N+1}\left(\left(\frac{i-1}{i}\right)^{n}\frac{n+i-1}{i-1}\right)=\lim_{N\to\infty}\left(\frac{k}{k+N+1}\right)^n\frac{(n+k+N)!}{(k+N)!}\cdot\frac{(k-1)!}{(n+k-1)!}.
    \end{align}
   Multiplying by $\pi_{n,k}$,
     \begin{align}
     A_{n,k,N}:=\pi_{n,k} \left(\frac{k}{k+N+1}\right)^n\frac{(n+k+N)!}{(k+N)!}\frac{(k-1)!}{(n+k-1)!}= \left(\frac{1}{k+N+1}\right)^n\frac{k!}{(k-n)!}\,\frac{(n+k+N)!}{(n+k-1)!}\,\frac{(k-1)!}{(k+N)!}. 
    \end{align}
    Expressing with binomial coefficients,
    \begin{align}
     A_{n,k,N}= \left(\frac{1}{k+N+1}\right)^n\frac{(n+k+N)!}{(k+N)!}\frac{\binom{k}{n}}{\binom{k+n-1}{n}}.
     \end{align}
Thus, for $k\ge n$,
     \begin{align}\label{eq: frac.binomials}
P(\tau^{(n)}\le k)=\lim_{N\to\infty}A_{n,k,N}=\frac{\binom{k}{n}}{\binom{k+n-1}{n}}\lim_{N\to\infty}\left(\frac{1}{k+N+1}\right)^n\frac{(n+k+N)!}{(k+N)!}.
    \end{align}
    Now use the fact that by Stirling's formula,
    if $a>0$ is fix and $u\to\infty$ then
    $\frac{(a+u)!}{u!}\sim u^a$ to conclude that
   the limit on the right hand side tends to one as $N\to\infty$ (and $k,n$ are fixed).
    Therefore, for $k\ge n$,
    \begin{align}\label{eq: frac.binomials.final}
         P(\tau^{(n)}\le k)=\frac{
    \binom{k}{n}
    }{
    \binom{k+n-1}{n}
    },
    \end{align}
   as claimed.
\end{proof}
Next, we show that $\tau^{(n)}$ scales with $n^2$, more precisely, we prove that $\tau^{(n)}/n^2$ has a limiting distribution as $n\to\infty$. 

Recall that the reciprocal of an exponential variable is called \emph{inverse exponential distribution}; we will denote this law by $\mathsf{IE}(\lambda$), when the original variable is $\mathsf{Exp}(\lambda)$-distributed.
\begin{theorem}\label{thm:IE}
As $n\to\infty$, the law of 
$\tau^{(n)}/n^2$ converges to $\mathsf{IE}(1)$.  Thus, the limiting distribution function for positive $x$ values is $F(x)=e^{-1/x}$.
\end{theorem}
\begin{remark}
    Although the $\mathsf{IE}(1)$ distribution has infinite mean, it has a unique mode which is $1/2$ (see Figure \ref{fig: mode}), hence the most likely value of $\tau^{(n)}$ is around $n^2/2$ for large $n$.
$\hfill\diamond$\end{remark}
\begin{figure}[ht]
\centering
\includegraphics[width=0.5\textwidth]{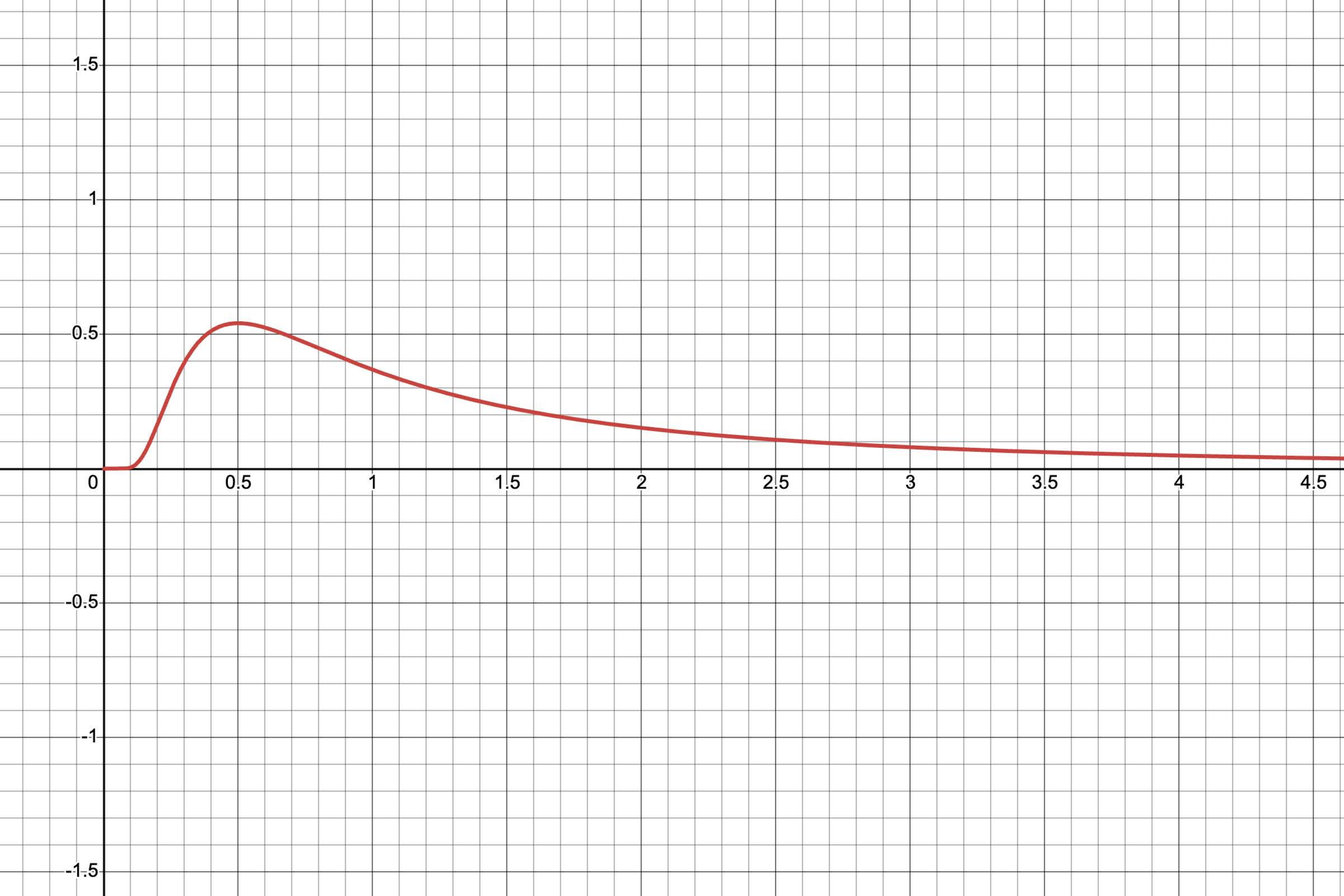}
\caption{The density of the limiting IE(1) law}
\label{fig: mode}
\end{figure}
\begin{proof}
By \eqref{eq: frac.binomials.final}, for $k\ge n$,
\begin{align}
P(\tau^{(n)}\le k)=\frac{k!/(k-n)!}{(k-1+n)!/(k-1)!}.
    \end{align}
    A little algebra yields the more useful form
    \begin{align}
        P(\tau^{(n)}\le k)=\left[\left(1+\frac{n}{k-1}\right)\left(1+\frac{n}{k-2}\right)...\left(1+\frac{n}{k-n+1}\right)\right]^{-1}.
    \end{align}
    In the rest of the proof we  are going to assume that $k=k(n)$ and $n=o(k)$, and use the elementary fact  that
    \begin{align}\label{eq: cd}
        \log (d+1)-\log  c\le \sum_c^d \frac{1}{i} \le \log  d-\log  (c-1),\ 1<c<d,
    \end{align}
    along with a Taylor expansion.
We thus have
\begin{align*} &\left(1+\frac{n}{k-1}\right)\left(1+\frac{n}{k-2}\right)...\left(1+\frac{n}{k-n+1}\right)=\\
 &\exp\left[\frac{n}{k-1}+\frac{n}{k-2}+...+\frac{n}{k-n+1}+\mathcal{O}\left(\frac{n^3}{(k-n+1)^2}\right)\right]=\exp\left(n\sum^{k-1}_{k-n+1}\frac{1}{i}+\mathcal{O}\left(\frac{n^3}{(k-n+1)^2}\right)\right),
   \end{align*}
   and since  by \eqref{eq: cd}, 
   $$\left(1+\frac{n-1}{k-n+1}\right)^n=\left(\frac{k}{k-n+1}\right)^n\le \exp\left(n\sum^{k-1}_{k-n+1}\frac{1}{i}\right) \le\left(\frac{k-1}{k-n}\right)^n=\left(1+\frac{n-1}{k-n}\right)^n,$$ it follows that
   for $k(n):=\gamma n^2, \gamma >0,$
   \begin{align} \lim_{n\to\infty}\left(1+\frac{n}{k-1}\right)\left(1+\frac{n}{k-2}\right)...\left(1+\frac{n}{k-n+1}\right)=e^{1/\gamma}.
   \end{align}
   Hence, 
   $ \lim_{n\to\infty}P(\tau^{(n)}\le \gamma n^2)=e^{-1/\gamma},$ as claimed.
\end{proof}
\subsection{The time of first detachment}\label{subs: first.det}
Consider now the {\it first} detachment time $\hat\tau^{(n)}=\min\{k\ge 1\mid L_{k}=n\}.$ The next theorem shows that the correct scaling for this time is between $n^2/\log n$ and $n^2.$
\begin{theorem} [Scaling for the first detachment time]\label{thm: quadratic}  Let $(k_n)$ be a  sequence  satisfying $k_n\ge n$.
    \begin{enumerate}
    \item[(i)] Let $n^2=\mathcal{O}(k_n)$. Then $$\liminf_n P(\hat \tau^{(n)}\le k_n)\ge \liminf_n P(L_{k_{n}}=n)>0.$$ In particular, $\hat \tau^{(n)}/n^2$ does not converge to infinity in law.
    \item[(ii)] Let $k_n\le \dfrac{n^2}{c\log n},\ c>4$, or more generally, assume that $(n-1)^2/(2k_n)-\log k_n\to\infty$. Then
    $$\lim_{n\to\infty} P(\hat \tau^{(n)}\le k_n)=0,$$ and so for $k_n=o\left( \dfrac{n^2}{\log n}\right),$ one has $\hat \tau^{(n)}/k_n\to\infty$ in law.
    \item[(iii)] (large deviations) For $c>1$, 
    $$ -\frac{1}{2c}-\,\frac{1}{6c^2}\le\liminf_{n\to\infty}\frac{1}{n}\log P(\hat\tau^{(n)}\le cn)
\le\limsup_{n\to\infty}\frac{1}{n}\log P(\hat\tau^{(n)}\le cn)\le -\frac{1}{2c}.$$
\end{enumerate}
\end{theorem}
\begin{proof} The proof is based on \eqref{eq: birthday}.
\begin{itemize}
\item[(i)] The statement follows immediately from \eqref{eq: birthday}.
\item[(ii)] By the union bound, $$P(\hat \tau^{(n)}\le k_n)\le \sum_{k=1}^{k_{n}} P(L_k=n)=\sum_{k=1}^{k_{n}}\frac{k!}{(k-n)!k^n}.$$
Using \eqref{eq: birthday}, one has
 \begin{align}\label{eq: rough}
 P(\hat \tau^{(n)}\le k_n)&<\sum_{k=1}^{k_{n}}\exp\{-(n-1)n/(2k)\}<k_n\exp\{-(n-1)^2/(2k_n)\}\nonumber\\&=
\exp\{\log k_n-(n-1)^2/(2k_{n})\},
\end{align}
and the statement follows.
\item[(iii)] The lower estimate follows from $P(\hat \tau^{(n)}\le cn)\ge  P(L_{cn}=n)$ along with \eqref{eq: birthday}. The upper bound follows from \eqref{eq: rough}.
\end{itemize}
The proof is complete.
\end{proof}
\begin{remark}
The final upper bound in (ii) was rough, but refining the estimate for the last sum provides minimal improvement. The primary issue is the overly crude union bound applied at the beginning of the argument in (ii).$\hfill\diamond$\end{remark}
    
\section{Scaling limit of times in a state of detachment --- quadratic scaling}\label{sec:scalinglimit}
The following result establishes a scaling limit with quadratic (in $n$) scaling of times.
\begin{theorem}[Scaling limit]\label{thm: scaling.limit.state.det} For $n\ge 1$
define the $0-1$-valued (non-Markovian) stochastic process $X^{(n)}$ on $[1,\infty)$ by $$X^{(n)}_t:=\mathds{1}(L_{tn^2}=n)=\mathds{1}(\text{state of detachment at time\ }tn^2),$$  and note that $X^{(n)}$ is c\`{a}dl\`{a}g by the way we extended the $n$-detachment process for real times. Then, as $n\to\infty$, the processes $X^{(n)}$ have a limit in the sense of finite dimensional distributions (fidis) to a process with c\`{a}dl\`{a}g paths. 
\end{theorem}
\begin{remark}[Lack of tightness]
For the laws of $\{X^{(n)},n\ge 1\}$ one cannot hope to have tightness  in the Skorokhod space $D([1,\infty))$, or even in $D([1,T])$ (and thus convergence in the sense of the laws of the processes).
The intuitive explanation is as follows. Let $k_1 = c n^2$ and $k_2 = d n^2$ with fixed constants $1 < c < d < T$, and define
$l := k_2 - k_1 = (d-c)n^2.$
Consider
\[
E_n(c,d)
:=P^{(n)}(L_{k_{2}}=n\mid L_{{k_{1}}}=n)= \left(\frac{k_1}{k_2}\right)^n
   {}_2F_0\!\left(-n,\,-l;\,-;\,\frac{1}{k_1}\right).
\]
Although the convergence
$E_n(c,d)\xrightarrow[n\to\infty]{} \exp\!\left(-\dfrac{d-c}{2d^2}\right) $
will be established in \eqref{eq: straight} below,  the error bound in \eqref{eq: conv.but.nonunif} is \underline{not} uniform on $\{(c,d):1<c<d<T\}$, because $d-c$ can be arbitrarily small.

\medskip
The lack of tightness means that certain properties (such as for instance the first hitting time of $1$) of the limiting  c\`{a}dl\`{a}g process may be very different from those of all the approximating processes. Indeed, however large $n$ is taken to be, the small scale fluctuations of the approximating processes will  differ from those of the limiting one.
$\hfill \diamond$
\end{remark}
Before turning to the proof we need some identities.
\subsection{Collecting some basic identities}\label{subsec: identities}
The following identities will be useful (and proven, if not shown before) in this section. They are collected here for convenience. Below $k_1=k_1(n)= c n^2,k_2=k_2(n)= d n^2,k_3=k_3(n)=hn^2$ and $1\le c<d<h.$
\begin{itemize}
\item[(1a)] $P^{(n)}(L_{k_1}=n)=\dfrac{(k_1)_n}{k_1^{n}}$
\item[(1b)] $\lim_{n\to\infty} P(L_{k_1}=n)=\exp(-1/(2c))$
\item[(2a)] $P^{(n)}(L_{k_{1}}=n, L_{{k_{2}}}=n)=\dfrac{(k_1)_n}{k_2^n}\, {}_2F_0(-n,k_1-k_2;-;1/k_1)$
\item[(2b)] $\lim_{n\to\infty}P^{(n)}(L_{k_{2}(n)}=n, L_{{k_{1}(n)}}=n)=\exp\left(-1/(2c)-\dfrac{d-c}{2d^{2}}\right)$
\item[(3a)] $P^{(n)}(L_{k_{2}}=n\mid L_{{k_{1}}}=n)= \left(\dfrac{k_1}{k_2}\right)^{\!n}
   \, {}_2F_0\!\left(-n,\,-l;\,-;\,\frac{1}{k_1}\right)$
\item[(3b)] $\lim_{n\to\infty}P^{(n)}(L_{k_{2}(n)}=n\mid L_{{k_{1}(n)}}=n)=\exp\left(-\dfrac{d-c}{2d^{2}}\right)$
\item[(4a)]  $P^{(n)}(L_{k_{1}}=L_{{k_{2}}}=L_{k_3}=n)=\dfrac{(k_1)_n}{k_2^n} {}_2F_0\!\left(-n,\,k_1-k_2;\,-;\,\dfrac{1}{k_1}\right)
\,\left(\dfrac{k_2}{k_3}\right)^{\!n} {}_2F_0\!\left(-n,\,k_2-k_3;\,-;\,\frac{1}{k_2}\right)$
\item[(4b)] $\lim_{n\to\infty} P^{(n)}(L_{k_{1}}=L_{{k_{2}}}=L_{k_3}=n)=\exp\left\{-1/(2c)-\dfrac{1}{2}\left(\dfrac{d-c}{d^2}+\dfrac{h-d}{h^2}\right)\right\}$
\item[(5a)]
$P(L_{k_1}=L_{k_3}=n, L_{k_2}\neq n)=$
$$\frac{(k_1)_n}{k_3^n}\,\left[ {}_2F_0(-n,k_1-k_3;-;1/k_1)-
   \, {}_2F_0\!\left(-n,\,k_1-k_2;\,-;\,\frac{1}{k_1}\right)
\, {}_2F_0\!\left(-n,\,k_2-k_3;\,-;\,\frac{1}{k_2}\right)\right]$$
\end{itemize}

\subsection{Convergence of the finite dimensional distributions}\label{subsec: conv.fidis}
\begin{proof}
The convergence of the one-dimensional distributions follows from \eqref{eq: asympt.log}. For the multidimensional distributions we use an inductive argument.
Let $v>1$ and $0<t_1<t_2<...<t_v$ and consider the events 
$E_1,...,E_v$  where each $E_i=E_{n,i}$ equals either $\{L_{t_{i}n^{2}}=n\}$ or $\{L_{t_{i}n^{2}}\neq n\}.$ Since $X^{n}_t$ can only take two values, it is sufficient to  show that
$P^{(n)}(E_1E_2...E_v)$ converges as $n\to\infty$. For that, we first assume that $E_i=E_{n,i}=\{L_{t_{i}n^{2}}=n\}$ for all $i=1,...,v$. By the Markov property,
$$P^{(n)}(E_1E_2...E_v)=P^{(n)}(E_1)\cdot P^{(n)}(E_2\mid E_1)\cdot P^{(n)}(E_3\mid E_2)\cdot...\cdot P^{(n)}(E_v\mid E_{v-1}),$$
where we also used that given detachment at $t_i n^2$ ($L_{t_{i}n^{2}}=n$)  it  is irrelevant in exactly which buses the passengers are located at $t_in^2$, as far as detachment at the next time  ($L_{t_{i+1}n^{2}}=n$) is concerned.
Therefore, to finish the argument in this case, it is sufficient to show the convergence of $P^{(n)}(E_{i+1}\mid E_{i})$. 

Let $k_1<k_2.$
Recall from Subsection \ref{subsec: identities}
that
\begin{align*}
P^{(n)}(L_{k_{2}}=
    n\mid L_{{k_{1}}}=n)= \left(\frac{k_1}{k_2}\right)^{\!n}
   \, {}_2F_0\!\left(-n,\,-l;\,-;\,\frac{1}{k_1}\right).
\end{align*}
We now show that if $k_1=k_1(n)= c n^2,k_2=k_2(n)= d n^2$
with $0<c<d$ then
\begin{align}\label{eq: straight}
    \lim_{n\to\infty}P^{(n)}(L_{k_{2}(n)}=
    n\mid L_{{k_{1}(n)}}=n)=\exp\left(-\frac{d-c}{2d^{2}}\right).
\end{align}
Indeed, define
$l=l(n) := k_2(n) - k_1(n) = (d - c)\,n^2,$
and
\[
S_n
:= \left(\frac{k_1}{k_2}\right)^{n}
{}_2F_0\!\left(-n,-l;\,-;\,\frac{1}{k_1}\right)
= \left(\frac{k_1}{k_2}\right)^{n}
\sum_{r=0}^{n}\binom{n}{r}\frac{(l)_r}{k_1^{\,r}},
\] where $k_1,k_2,l$ depend on $n$.
Letting \(x := \dfrac{l}{k_1} = \dfrac{d-c}{c}\),
we may rewrite the sum as
\[
\sum_{r=0}^{n}\binom{n}{r}\frac{(l)_r}{k_1^{\,r}}
= \sum_{r=0}^{n}\binom{n}{r}x^{r}\frac{(l)_r}{l^r}
= (1+x)^n\sum_{r=0}^{n}\binom{n}{r}\left(\frac{x}{1+x}\right)^r\left(\frac{1}{1+x}\right)^{n-r}\frac{(l)_r}{l^r}=(1+x)^{n}\,\mathbb{E}\!\left[\frac{(l)_R}{l^{R}}\right],
\]
where \(R =R_n\sim \mathsf{Bin}(n,\theta)\) with success probability
\[
\theta: = \frac{x}{1+x} =\left(\frac{d-c}{c}\right)\big/\left(\frac{d}{c}\right)= \frac{d-c}{d},
\]
and $\mathbb{E}=\mathbb{E}_{\theta}$ is the corresponding expectation.
Since \(\dfrac{k_1}{k_2} =\dfrac{k_1-k_2}{k_2}+1=\dfrac{c-d}{d}+1=\dfrac{c}{d}= \dfrac{1}{1+x}\), it follows that
\(
S_n = \mathbb{E}\!\left[\dfrac{(l)_R}{l^{R}}\right].
\)
Next, use an expansion for the logarithm when $r\le n\wedge l$. Using the Lagrange remainder term (denoted by $L$ here),
\[
\log\frac{(l)_r}{l^r}
= \sum_{j=0}^{r-1} \log\!\left(1 - \frac{j}{l}\right)=\sum_{j=0}^{r-1}\left[-\frac{j}{l}-L(j/l)\right]=-\frac{r(r-1)}{2l}-\sum_{j=0}^{r-1}L(j/l),
\] where $0\le L(j/l)\le \frac{1}{2(1-(j/l))^2}\left(\frac{j}{l}\right)^2=\frac{j^2}{2(l-j)^2},$
that is,
\begin{align}\label{eq: conv.but.nonunif}
&\left|\log\frac{(l)_r}{l^r}+\frac{r(r-1)}{2l}\right| \le  \sum_{j=0}^{r-1} \frac{j^2}{2(l-j)^2}\le \sum_{j=0}^{r-1} \frac{j^2}{2((c-d)n^2-n)^2}\nonumber\\&\qquad=\frac{(r-1)r(2r-1)}{12((c-d)n^2-n)^2}\le \frac{n^3}{6((c-d)n^2-n)^2}:=e(c,d,n).
\end{align}
 (Clearly, for $c,d$ fixed, $e(c,d,n)=\mathcal{O}(1/n).$)
Hence, for $r\le l,$  
\[ \exp\left\{-\frac{r(r-1)}{2l} - e(c,d,n)\right\}\le\frac{(l)_r}{l^r}\le \exp\left\{-\frac{r(r-1)}{2l}\right\} \]
and thus
\begin{align}\label{eq:thus}
\exp\left(- e(c,d,n)\right)\mathbb{E}\!\left[\exp\!\left(-\frac{R_n(R_n-1)}{2l}\right)\right]\le S_n=
\mathbb{E}\!\left[\dfrac{(l)_R}{l^{R}}\right]\le \mathbb{E}\!\left[\exp\!\left(-\frac{R_n(R_n-1)}{2l}\right)\right].\end{align}
Since \(R_n \sim \mathsf{Bin}\!\left(n,\frac{d-c}{d}\right)\), one can couple all $R_n$ in the obvious way and then
one has \(\lim_n \dfrac{R_n}{n} =\theta=  \dfrac{d-c}{d}\) a.s. by the Strong Law of Large Numbers.
Hence, almost surely,
\[
\frac{R_n(R_n-1)}{2l}
= \frac{R_n^2 - R_n}{2(d-c)n^2}
\;\longrightarrow\;
\frac{\theta^2}{2(d-c)} = \frac{d-c}{2d^2},
\]
and this and \eqref{eq:thus}, along with by bounded convergence implies that
\[
\lim_{n\to\infty} S_n
= \exp\!\left(-\frac{d-c}{2d^2}\right),
\]
completing the proof of \eqref{eq: straight}.

Returning now to the proof of the theorem, next, assume that $E_i=E_{n,i}=\{L_{t_{i}n^{2}}\neq n\}$ for \emph{exactly one} index $i$. For simplicity assume that this index is $i=1.$ Then we are done by observing that
$$P^{(n)}(E_1E_2...E_v)=P^{(n)}(E_2...E_v)-P^{(n)}(E^c_1E_2...E_v).$$ Next, assume that $E_i=E_{n,i}=\{L_{t_{i}n^{2}}\neq n\}$ for \emph{exactly two} indices $i$; and suppose for simplicity that these are $i=1,2$. Similarly as before, we may write
$$P^{(n)}(E_1E_2...E_v)=P^{(n)}(E_3...E_v)-P^{(n)}(E^c_1E_2...E_v)-P^{(n)}(E_1E^c_2...E_v)-P^{(n)}(E^c_1E^c_2...E_v),$$
and we are done again because of the already proven cases. Continuing the argument inductively, we check that $P^{(n)}(E_1E_2...E_v)$ always has a limit as $n\to\infty$.
\end{proof}
\begin{remark}[Transition from non-detached to detached]\label{rem: worth}
    It is worth recording also that if $k_1<k_2$ then letting \( l := k_2 - k_1 > 0 \)
    \begin{align}
P^{(n)}(L_{k_{2}}=
    n\mid L_{{k_{1}}}\neq n)=:P(B\mid A)=\frac{P(B)-P(A^c B)}{P(A)}&=\nonumber\\
    \frac{p_n(k_2)-\pi(n,k_1,k_2-k_1)}{1-p_n(k_1)}&=
    \frac{p_n(k_2)-\left(\frac{k_1}{k_2}\right)^{\!n}
   \, {}_2F_0\!\left(-n,\,-l;\,-;\,\frac{1}{k_1}\right)}{1-p_n(k_1)}.
    \end{align}
Consequently,  if $k_1=k_1(n)= c n^2,k_2=k_2(n)= d n^2$
with $0<c<d$ then
$$\lim_{n\to\infty}P^{(n)}(L_{k_{2}}=
    n\mid L_{{k_{1}}}\neq n)=
\frac{\exp\left(-\dfrac{1}{2d}\right)-\exp\!\left(-\dfrac{1}{2c}-\dfrac{d-c}{2d^2}\right)}{1-\exp(-1/(2c))}.$$
 \end{remark}   

\subsection{Checking that the limit given by the fidis  has a càdlàg version}

We are going to show that the limiting process (in the sense of fidis) has a version with values in the Skorokhod-space of cadlag functions. For that it will be enough to check the sufficient condition \eqref{eq:cadlag.for.01} given in our Appendix B. The first half of it is that
$\lim_{h\downarrow 0}P(X_{t+h}\neq X_t)=0.$
This is clearly satisfied since
$$\lim_{n\to\infty}P^{(n)}(L_{k_{2}(n)}=n\mid L_{{k_{1}(n)}}=n)=\exp\left(-\dfrac{d-c}{2d^{2}}\right),$$
and so for the limiting process $X$,
$P(X_c=1\mid X_d=1)=\exp\left(-\dfrac{d-c}{2d^{2}}\right),\ 1\le c<d,$ which tends to one if $d\downarrow c.$
Similarly, by Remark \ref{rem: worth}
if $k_1=k_1(n)= c n^2,k_2=k_2(n)= d n^2$
with $0<c<d$ then
$$P(X_c=1\mid X_d=0)=\lim_{n\to\infty}P^{(n)}(L_{k_{2}}=
n\mid L_{{k_{1}}}\neq n)=
\frac{\exp\left(-\dfrac{1}{2d}\right)-\exp\!\left(-\dfrac{1}{2c}-\dfrac{d-c}{2d^2}\right)}{1-\exp(-1/(2c))},$$ which tends to zero if $d\downarrow c.$

Therefore, it only remains to check the second half, involving the probabilities 
$P(L_{k_1}=L_{k_3}=n, L_{k_2}\neq n)$ and  $P(L_{k_1}\neq n,L_{k_2}=n,L_{k_3}\neq n)$.

\bigskip
\underline{Checking the first probability.} Using Remark \ref{rem: worth}, one has
\begin{align}
\lim_{n\to\infty}P^{(n)}(L_{k_1}=L_{k_3}=n, L_{k_2}\neq n)
&=\lim_{n\to\infty}P(L_{k_1}=L_{k_3}=n)-\lim_{n\to\infty}P(L_{k_1}=L_{k_2}=L_{k_3}= n)\\
&= e^{-1/(2c)}\left[\exp\left\{-\dfrac{h-c}{2h^2}\right\}-\exp\left\{-\dfrac{1}{2}\left(\dfrac{d-c}{d^2}+\dfrac{h-d}{h^2}\right)\right\}\right].\nonumber
\end{align}
We first claim that for all  $0<c<d<h,$
\begin{align}\label{eq:Chentsov.1}
\lim_{n\to\infty} P(L_{k_1}=L_{k_3}=n, L_{k_2}\neq n)
< K(h-c)^{1+\alpha},
\end{align}
with $K=0.6$ and $\alpha=1/2$, that is,
\begin{align}\label{eq:Chentsov.1.writtenout}
e^{-\frac{1}{2c}}
\!\left[
e^{-\frac{h-c}{2h^2}}
- e^{-\frac{1}{2}\!\left(\frac{d-c}{d^2}+\frac{h-d}{h^2}\right)}
\right]
< 0.6\cdot (h-c)^{3/2}.
\end{align}
(So $K,\alpha$ do not depend on $T$.)
To verify \eqref{eq:Chentsov.1.writtenout}, let
\[
A:=\frac{h-c}{2h^2},\qquad
B:=\frac{1}{2}\!\left(\frac{d-c}{d^2}+\frac{h-d}{h^2}\right)
= \frac12\,(d-c)\!\left(\frac{1}{d^2}-\frac{1}{h^2}\right)+A.
\]
Thus \(B>A>0\) and
\[
e^{-A}-e^{-B}=\int_A^B e^{-t}\,dt \le (B-A)\,e^{-A}\le B-A.
\]
Therefore, the lefthand side of \eqref{eq:Chentsov.1.writtenout} satisfies
\[
e^{-\frac{1}{2c}}\!\left[e^{-A}-e^{-B}\right]\le e^{-\frac{1}{2c}}(B-A).
\]
Now write \(h=c+\varepsilon\) with \(\varepsilon:=h-c>0\), and parametrize
\(d=c+\tau\varepsilon\) with \(\tau\in(0,1)\). A direct calculation yields 
that
\[
B-A
=\frac{\varepsilon^{2}\,\tau(1-\tau)\,\bigl(2c+\varepsilon(1+\tau)\bigr)}
{2\,(c+\tau\varepsilon)^{2}\,(c+\varepsilon)^{2}}.
\]
Dividing by \((h-c)^{3/2}=\varepsilon^{3/2}\) one obtains
\begin{equation}\label{eq:ast}
\frac{e^{-\frac{1}{2c}}\!\left[e^{-A}-e^{-B}\right]}{(h-c)^{3/2}}
\;\le\;
e^{-\frac{1}{2c}}\,
\frac{\varepsilon^{1/2}\,\tau(1-\tau)\,\bigl(2c+\varepsilon(1+\tau)\bigr)}
{2\,(c+\tau\varepsilon)^{2}\,(c+\varepsilon)^{2}}
=: \Phi(c,\varepsilon,\tau).
\end{equation}
We need to bound \(\Phi\) uniformly over the domain \(0<c<d<h\) (equivalently, over
 \(\varepsilon>0\), \(\tau\in(0,1)\)).
Use \(\tau(1-\tau)\le \tfrac14\), \(2c+\varepsilon(1+\tau)\le 2(c+\varepsilon)\), and $\tau,\varepsilon>0,$ to obtain
\[
\Phi(c,\varepsilon,\tau)
\le e^{-\frac{1}{2c}}\,
\frac{\varepsilon^{1/2}}{4}\cdot
\frac{2(c+\varepsilon)}{2\,(c+\tau\varepsilon)^{2}(c+\varepsilon)^{2}}
= e^{-\frac{1}{2c}}\,
\frac{\varepsilon^{1/2}}{4\,(c+\tau\varepsilon)^{2}(c+\varepsilon)}
\le e^{-\frac{1}{2c}}\,
\frac{\varepsilon^{1/2}}{4\,c^{2}(c+\varepsilon)}
=: \Psi_{c}(\varepsilon).
\]
Maximize \(\Psi_{c}(\varepsilon)\) over \(\varepsilon>0\), by setting \(x=\sqrt{\varepsilon}\), and writing
\(
\Psi_{c}(\varepsilon)= \exp\left({-\dfrac{1}{2c}}\right)\,\dfrac{x}{4\,c^{2}(c+x^{2})}:
\)
the maximum occurs at $x^{2}=c,$ hence
\[
\sup_{\varepsilon>0}\Psi_{c}(\varepsilon)
= e^{-\frac{1}{2c}}\cdot \frac{\sqrt{c}}{4\,c^{2}(c+c)}
= \frac{e^{-\frac{1}{2c}}}{8\,c^{5/2}}.
\]
Since $\sup_{c>0}\frac{e^{-\frac{1}{2c}}}{8\,c^{5/2}}<0.6$,  \eqref{eq:ast} yields that
\[
\sup_{0<c<d<h}\;
\frac{e^{-\frac{1}{2c}}\!\left[e^{-A}-e^{-B}\right]}{(h-c)^{3/2}}
\le \frac{e^{-\frac{1}{2c}}}{8\,c^{5/2}}<0.6,
\]
and \eqref{eq:Chentsov.1.writtenout} is verified.

\bigskip
\underline{Checking the second probability.} 
We have
\begin{align*}
    P^{(n)}(L_{k_1}\neq n,L_{k_2}=n,L_{k_3}\neq n)=P^{(n)}(L_{k_1}\neq n)P(L_{k_2}= n\mid L_{k_1}\neq n)P^{(n)}(L_{k_3}\neq n\mid L_{k_2}=n,L_{k_1}\neq n).
\end{align*}
But the very last conditioning on $L_{k_1}\neq n$ can be dropped, that is
\begin{align}
    P^{(n)}(L_{k_1}\neq n,L_{k_2}=n,L_{k_3}\neq n)=P^{(n)}(L_{k_1}\neq n)P^{(n)}(L_{k_2}= n\mid L_{k_1}\neq n)P^{(n)}(L_{k_3}\neq n\mid L_{k_2}=n),
\end{align}
because if at $k_2$ the process is in a state of detachment then it does not matter where exactly the passengers are located at that time.
Hence, similarly to the first part, it follows that
\begin{align}
&\lim_{n\to\infty}P^{(n)}(L_{k_1}\neq n,L_{k_2}=n,L_{k_3}\neq n)\nonumber \\
&\qquad =\left[\exp\left(-\dfrac{1}{2d}\right)-\exp\!\left(-\dfrac{1}{2c}-\dfrac{d-c}{2d^2}\right)\right]\left[1-\exp\left(-\frac{h-d}{2h^2}\right)\right]=:G(c,d,h).
\end{align}
For \(1\le c<d<h\le T\), we are going to show the existence of a constant \(K(T)>0\) such that 
\begin{align}\label{eq:bit.complicated}
F(c,d,h):=\frac{G(c,d,h)}{(h-c)^{3/2}}\le K(T).
\end{align}
(So $\alpha$ does not depend on $T$ but $K$ does.)
First note that for \(u,v\ge0\), \(e^{-u}-e^{-(u+v)}=e^{-u}(1-e^{-v})\le v\,e^{-u}\), and so with
\[
u=\frac1{2d},\quad
v=\Bigl(\frac1{2c}-\frac1{2d}\Bigr)+\frac{d-c}{2d^2}
=\frac{(d-c)(d+c)}{2cd^2},
\]
we get
\begin{align*}
\bigl[e^{-\frac1{2d}}-e^{-\frac1{2c}-\frac{d-c}{2d^2}}\bigr]
\le e^{-\frac1{2d}}\cdot\frac{(d-c)(d+c)}{2cd^2}.
\end{align*}
Since we also have  $\bigl[1-e^{-\frac{h-d}{2h^2}}\bigr]\le \frac{h-d}{2h^2},$ it follows that
\[
F(c,d,h)\;\le\;
e^{-\frac1{2d}}\,
\frac{(d-c)(h-d)(d+c)}{4\,c\,d^{2}\,h^{2}}\,(h-c)^{-3/2}.
\]
Note that $c,d,h>1,$
$e^{-\frac1{2d}}\le e^{-\frac1{2T}},$
while 
$(d-c)(h-d)\le \dfrac{(h-c)^{2}}{4}$
and  $d+c\le 2h$. Thus
\[
F(c,d,h)
\le e^{-\frac1{2T}}\,
\frac{\frac{(h-c)^{2}}{4}\cdot 2h}{4\,c\,d^{2}\,h^{2}}\,(h-c)^{-3/2}
= e^{-\frac1{2T}}\,
\frac{(h-c)^{1/2}}{32\,c\,d^{2}\,h}\le 
\frac{e^{-\frac1{2T}}\cdot \sqrt{T}}{32}
=:K(T),
\]
which verifies \eqref{eq:bit.complicated}.

\section{Further scaling limit results}\label{sec: further}
\subsection{Detachment states in expectation --- quadratic scaling with logarithmic correction}\label{subsec: critical window}

A subtle yet crucial distinction emerges in this subsection: while quadratic scaling is necessary to define the process's scaling limit (cf. Theorem \ref{thm: scaling.limit.state.det}), the behavior of the \emph{expected number of detached states} demonstrates that sub-quadratic scaling is appropriate for this specific metric.

Let $X_t:=\mathds{1}(L_t=n)$ for $t\ge n.$ We recall from \eqref{eq:pi.clear} that 
$\pi_{n,k}:=P(L_k=n)=\frac{(k)_n}{k^{n}},$ and
define 
\begin{align}
e(n,k):=E^{(n)}\left(\sum_{j=1}^{k(n)}X_j\right)&=\sum_{j=1}^{k(n)}p_j=\sum_{j=1}^{k(n)}P(L_j=n)=\sum_{j=1}^{k(n)}j^{-n}j(j-1)\dots (j-n+1)
\\&=\sum_{j=n}^{k(n)}j^{-n}j(j-1)\dots (j-n+1)=\sum_{j=n}^{k(n)}j^{-n}(j)_n,\nonumber
\end{align}
which is the expected number of detachment states up to $k(n)$. 

The behavior of $e(n,k(n))$  for large $n$ and for various choices of $k(n)$ is examined next. 
\begin{theorem}[Expected number of detachment states]\label{prop: expect}
 If 
 $$k(n)\sim\frac{n^{2}}{c\log n},\ n\gg 1,$$ and $c>4$
 then $e(n,k)\to 0$ as $n\to\infty$, i.e. $\sum_{j=1}^{k(n)}X_j\to 0$ in $L^1$ (cf. Theorem \ref{thm: quadratic}(ii)), while 
 if $c<4$ then $e(n,k)\to \infty.$ 
Furthermore, if $y>0$ and 
\begin{align}\label{eq: crit.rate} k(n)=k(n,y):=\frac{n^{2}}{2\,(2\log n-2\log\log n+y)}\sim \frac{n^{2}}{4\log n},
\end{align}
then 
\begin{align}\label{eq: c(y)}
\lim_{n\to\infty} e(n,k(n)) = \frac{e^{-y}}{8}=:c(y).
\end{align}
\end{theorem}
\begin{corollary}[The critical time window]
Expanding $k(n,y)$ as
 \begin{align}
 k(n,y)&= \frac{n^{2}}{4\log n}+   \frac{n^{2}\log\log n}{4(\log n)^2} - \frac{n^{2}y}{8(\log n)^2} + \mathcal{O}\left(\frac{n^2(\log\log n)^2}{(\log n)^3}\right)\nonumber\\&= \frac{n^{2}}{4\log n}\left\{1+   \frac{\log\log n-y/2}{\log n}+ \mathcal{O}\left(\left[\frac{(\log\log n)}{(\log n)}\right]^2\right)\right\},
 \end{align}
 we  conclude that for $c>1/8$ and
\begin{align*}
k(n,c):=\frac{n^{2}}{4\log n}\left\{1+   \frac{\log\log n+\log\sqrt{8c}}{\log n}+ \mathcal{O}\left(\left[\frac{(\log\log n)}{(\log n)}\right]^2\right)\right\}
\end{align*}
the relation $\lim_n e(n,k(n,c))=c$ holds, that is one needs  $k(n,c)$ amount of time to have $c$ detachment events on average with $n\gg 1$ passengers. Loosely speaking, with $0<\mathrm{const}_1<\mathrm{const}_2$, the {\bf critical time window} is about $$\left[\frac{n^{2}}{4\log n}\left\{1+\frac{\log\log n+\mathrm{const}_1}{\log n}\right\},\frac{n^{2}}{4\log n}\left\{1+\frac{\log\log n+\mathrm{const}_2}{\log n}\right\}\right]$$ for large $n$, and the expected number of times in a state of detachment  there is about \\
$\frac{1}{8}\left(\exp(2\,\mathrm{const}_2)-\exp(2\,\mathrm{const}_1)\right)$. The size of the window around $n^2/(4\log n)$ is
$\sim\mathrm{const}\cdot\left(\frac{n}{\log n}\right)^2$. 
\end{corollary}
\begin{proof}(of Theorem \ref{prop: expect}:)

We first consider the $k(n)\sim\dfrac{n^{2}}{c\log n}$ case.

\medskip 
\noindent\underline{Lower bound and $e(n,k)\to \infty$ for $c<4$.}

\medskip 
In order to analyze the behavior of 
$$e(n,k)=\sum_{j=n}^{k(n)}j^{-n}j(j-1)\dots (j-n+1)
=\sum_{j=n}^{k(n)}j^{-n}(j)_n,$$
we first focus on the term $j^{-n}(j)_n$ and  will use that for $1<a<b$,
\[
\int_{0}^{a}\log\!\left(1-\frac{x}{b}\right)\,\mathrm{d}x
= \; -a \;-\; (b-a)\,\log\Bigl(1-\frac{a}{b}\Bigr),
\qquad 0 \le a < b,
\] so for $n\ge 2$ and $j\ge n,$
\[0\ge F^{(n-1)}(j):=
\int_0^{\,n-1} \log\Bigl(1-\frac{x}{j}\Bigr)\,\mathrm{d}x
\;=\; -(n-1) \;-\; (\,j-n+1\,)\,\log\Bigl(1-\frac{n-1}{j}\Bigr),
\qquad j \ge n.
\]
Hence, using monotonicity, and comparing the sum with the integral, 
\[
\log(j^{-n}(j)_n)=\log(1-1/j)+...+\log(1-(n-2)/j)\ge F^{(n-1)}(j),\ j\ge n.
\]
Extend the definition of $F^{(n-1)}$ as \[F^{(n-1)}(x):=
\; -(n-1) \;-\; (\,x-n+1\,)\,\log\Bigl(1-\frac{n-1}{x}\Bigr)=-x\left[z \;+\; (\,1-z\,)\,\log\Bigl(1-z\Bigr)\right],
\qquad x > n-1,\] where $z:=\frac{n-1}{x},$ and note that in $e(n,k)$ the index starts at $j=n$.
Since $e^F$ is strictly increasing in $j$, we can use a sum-integral comparison once again: for  $k>n-1$,
\[ e(n,k)\ge \sum_{j=n}^k e^{F^{(n-1)}(j)}\ge 
\int_{n-1}^{k} e^{F^{(n-1)}(x)}\,\mathrm{d}x.
\]
Hence, it suffices to verify that if
\[
I(n)\;=\;\int_{\,n-1}^{\,k(n)} \exp\!\big(F^{(n-1)}(x)\big)\,\mathrm{d}x,
\qquad 
\text{with\ }k(n)\sim\frac{n^{2}}{c\log n},\quad 0<c<4,
\]
then $\lim_{n\to\infty} I(n)=\infty.$
First, 
use the expansion on $0<z<1$:
$$z+(1-z)\log(1-z) \le \frac{z^2}{2}  + K_0z^{3},\ \text{as}\ z\to 0,$$
where $K_0$ is an absolute constant (in fact $K_0=1/2$),
yielding for $x>n-1$ that
\[
F^{(n-1)}(x) \ge -\frac{(n-1)^{2}}{2x} - K_0\!\frac{(n-1)^{3}}{x^{2}},\ \text{as}\ n\to\infty,\ \frac{n}{x}\to 0,
\]  hence
\[
e^{F^{(n-1)}(x)} \ge \exp\!\Bigl(-\tfrac{(n-1)^{2}}{2x}-\text{err}\Bigr),\ \text{as}\ \frac{n}{x}\to 0,
\]
where $\text{err}\le K_0\!\dfrac{(n-1)^{3}}{x^{2}}.$

Next, we break the region of integration in $I(n)$ to
$[n-1,\,k(n)/2)$ and $[k(n)/2,k(n)]$. Let us now consider the integral 
\[
I_2(n) \;:=\; \int_{k(n)/2}^{k(n)} \exp\!\Bigl(-\tfrac{(n-1)^{2}}{2x}-\text{err}\Bigr)\,\mathrm{d}x.
\]
Change variables and
let $C_n:=\tfrac{(n-1)^{2}}{2}$, write $u:=C_n/x$, hence $\mathrm{d}x=-C_n\,\mathrm{d}u/u^{2}$, and note that for the error term one has
$\text{err}\le 4K_0 \,\dfrac{(n-1)^{3}}{k(n)^{2}}\to 0$ as $n\to\infty.$
Consequently, as $n\to\infty,$
\[
I_2(n)\ge\int_{k/2}^{k} e^{-(C_n/x)-\text{err}}\,\mathrm{d}x
\sim \int_{k/2}^{k} e^{-C_n/x}\,\mathrm{d}x\sim C_n\int_{C_{n}/k}^{2C_{n}/k} \frac{e^{-u}}{u^{2}}\,\mathrm{d}u
\;\sim\; \frac{k^{2}}{C_{n}}\,e^{-C_{n}/k},
\]
where in the last step we used  the asymptotics 
$\int_{x}^{2x} \frac{e^{-u}}{u^2} du \sim \frac{e^{-x}}{x^2},\ x\to\infty$, and that $C_n/k=C_n/k(n)\to \infty.$

Similar calculation shows that with  $c_1:=c/2$ and $n\gg 1,$
\begin{align*}
I_1(n) \;:=\;\int_{\,n-1}^{\,k(n)/2} \exp\!\big(F^{(n-1)}(x)\big)\,\mathrm{d}x\le\; \int^{k(n)/2}_{n-1} \exp\!\Bigl(-\tfrac{(n-1)^{2}}{2x}\Bigr)\,\mathrm{d}x=C_n\int_{2C_{n}/k}^{(n-1)/2} \frac{e^{-u}}{u^{2}}\,\mathrm{d}u\\=C_n\int_{(\frac{n-1}{n})^{2}c\log n+o(1)}^{(n-1)/2} \frac{e^{-u}}{u^{2}}\,\mathrm{d}u\le C_n\int_{c_1\log n}^{(n-1)/2} \frac{e^{-u}}{u^{2}}\,\mathrm{d}u= \frac{n^{-c_{1}}}{c_{1}^{2}(\log n)^{2}} + \mathcal{O}\left(\frac{n^{-c_{1}}}{(\log n)^{3}}\right).
\end{align*}
It follows that $I_1(n)=o(I_2(n)),$ since $k(n)\sim\dfrac{n^{2}}{c\log n}$, $c_1=c/2$ and so
$$I_1(n)/I_2(n)=\mathcal{O}\left(n^{(c/2)-c_1}\frac{(n-1)^2}{n^{4}}\right)=\mathcal{O}(n^{-2}).$$
Therefore
\[
I(n)\;\sim\;I_2(n)\;\sim\; \frac{2\,k(n)^{2}}{n^{2}}\,
\exp\!\Bigl(-\frac{n^{2}}{2k(n)}\Bigr).
\]
Now, if $a_n:=\frac{n^{2}}{2k(n)}\to \infty$ then 
$I(n)\sim 
\dfrac{n^{2}}{2}
\dfrac{e^{-a_{n}}}{a_{n}^{2}}$
and since $k(n)\sim\dfrac{n^{2}}{c\log n}$, $a_n\sim \frac{c}{2}\log n$, 
and $e^{-a_{n}}=\exp(-\frac{c}{2}\log n+o(1))$, yielding
\begin{align}\label{eq: subtle}
I(n)\sim 
\frac{2n^{2}}{(c\log n)^{2}}
\frac{1}{\exp(\frac{c}{2}\log n+o(1))},
\end{align}
and if $c\neq 4$ then this implies that
\begin{align}\label{eq: I(n).asympt}
\quad
I(n)\;\sim\;\frac{2}{c^{2}}\,\frac{n^{\,2-\tfrac{c}{2}}}{(\log n)^{2}},
\qquad n\to\infty.\quad
\end{align}
It follows that  if $c<4$ then $I(n)\to\infty$ polynomially, and therefore $e(n,k)\to\infty$ too.

\bigskip
\noindent\underline{Upper bound and $e(n,k)\to 0$ for $c>4$.}

\medskip
After noticing that
\begin{align}\label{eq: notice}
 \sum_{j=n}^{k(n)}j^{-n}j(j-1)\dots (j-n+1)&=
 \sum_{j=n}^{k(n)}\left(1-\frac{1}{j}\right)\left(1-\frac{2}{j}\right)\dots \left(1-\frac{n-1}{j}\right)\nonumber\\ &\le\sum_{j=n}^{k(n)}\exp\left\{-\frac{1}{j}\left(1+2+\dots+n-1\right)\right\},
\end{align}
the rest of the argument is similar but much simpler than for the lower bound, leading again to the  conclusion  that  the righthand side of \eqref{eq: notice} is asymptotical with the expression on the righthand side of \eqref{eq: I(n).asympt}, but for $c>4$  it implies convergence to zero. 
We leave the details to the reader. 
\bigskip

\noindent\underline{Fine tuning for $k(n)\sim \dfrac{n^{2}}{4\log n}$.}

\medskip
We will use the same notation as before. Just like in  the previous parts of the proof, 
the  relation \eqref{eq: subtle} still holds, but now, as $c=4$, the asymptotic behavior depends on the $o(1)=a_n-2\log n=:b_n$ term there, and we obtain that
\begin{align}\label{eq: b_n}
I(n)\sim 
\frac{2}{(4\log n)^{2}}
\frac{1}{e^{b_n}},
\end{align}
and this converges to a positive finite limit if and only if $b_n+2\log (4\log n)$ does. Given that $a_n=2\log n+b_n=2\log n+\text{const}-2\log (4\log n)$ and $k(n)=\frac{n^{2}}{2a_n},$ we arrive at the choice
$$k(n):=\frac{n^{2}}{2(2\log n -2\log \log n+\text{const}-2\log 4)}=\frac{n^{2}}{2\,(2\log n-2\log\log n+y)}\sim \frac{n^{2}}{4\log n},$$ (with $y:=\text{const}-2\log 4$)
which, plugging into \eqref{eq: b_n}, gives the limit $e^{-y}/8.$
\end{proof}
\begin{remark}[Second moment formula]
Let, as before, $X_i:=\mathds{1}\{L_i=n\}$; let $M_2(n,k):=E\left(\sum_1^k X_i\right)^2$. The explicit formula for $M_2(n,k)$ below might be of interest.  By \eqref{eq: Hypergeom}, 
 for $0<k,l$,
\begin{align*}
\pi(n,k,l)=P^{(n)}(L_{k+l}= L_k=n)=\frac{(k)_n}{(k+l)^n}\sum_{r=0}^{\min(n, l)}\binom{n}{r}\frac{(l)_{r}}{k^{r}}=\frac{(k)_n}{(k+l)^n}\, {}_2F_0(-n,-l;-;1/k),
\end{align*} hence
\begin{align*}
M_2(n,k)
:=& \sum_{m=1}^{k} \frac{(m)_{n}}{m^{\,n}}
+ 2 \sum_{\substack{m,l\ge1 \\ m+l\le k}} 
\frac{(m)_{n}}{(m+l)^{\,n}}
\sum_{r=0}^{\min(n,l)} 
\binom{n}{r}\frac{(l)_{r}}{m^{\,r}}
\\
=&\sum_{m=1}^{k} \frac{(m)_{n}}{m^{\,n}}
+ 2 \sum_{\substack{m,l\ge1 \\ m+l\le k}} \frac{(m)_n}{(m+l)^n}\, {}_2F_0(-n,-l;-;1/m),
\end{align*}
where ${}_2F_0$ is the generalized hypergeometric function. $\hfill \diamond$
\end{remark}
\subsection{``Almost detachment'' --- waiting for super-linear times}
We now show the perhaps surprising result that although one has to wait roughly quadratic (in $n$) time for complete detachment, {\it any} super-linearly large time is sufficient for seeing ``almost complete'' detachment.
\begin{theorem}[Almost detachment along any sequence of super-linear times]
Assume that $n\ll k=k(n).$
Then, as $n\to\infty,$
\begin{align}\label{eq: both.converge}
   \frac{N_{k(n)}}{n}\to 1,\ 
    \frac{L_{k(n)}}{n}\to 1 \qquad\text{in probability}.
\end{align}
\end{theorem}
   \begin{proof}
(i)  For the first convergence, since $N_k\le n$ for all $k\ge 1$, we only need to show that for $\varepsilon>0,$
$\lim_n P^{(n)}\left(N_{k(n)}<(1-\varepsilon)n\right)=0.$
Using Stirling numbers of the second kind (see Appendix A), if
\[ a_{n,j}
:= k^{-n}\binom{k}{j}S(n,j)j!,
\qquad
1\le j\le n,
\]
then
\[
\mathbb{P}\!\left(N_k \le (1-\varepsilon)n\right)=
\sum_{j=1}^{(1-\varepsilon)n} a_{n,j},
\]
where the reason for having the factor $j!$ is that we consider labeled buses.
Let $C := n - N_k$ and
\[
Y := \sum_{1\le a<b\le n}
     \mathds{1}_{\{\text{passengers $a$ and $b$ sit in the same bus}\}}.
\]
If a bus has  $r\ge 1$ passengers, then its contribution to $C$ is $r-1$, while its
contribution to $Y$ is $\binom{r}{2}$.  Since $\binom{r}{2}\ge r-1$ for all
$r\ge1$ (for $r=1$ both vanish), we obtain
$C \;\le\; Y,$ and hence
\[
\mathbb{E}[C]
\;\le\; \mathbb{E}[Y]
= \frac{n(n-1)}{2k}.
\]
By Markov's inequality, and since $n\ll k,$
\[
\mathbb{P}\!\left(N_k \le (1-\varepsilon)n\right)
= \mathbb{P}(C\ge \varepsilon n)
\;\le\; \frac{\mathbb{E}[C]}{\varepsilon n}
\;\le\; \frac{n(n-1)}{2\varepsilon n k}
=\frac{n-1}{2\varepsilon k}
\longrightarrow 0.
\]

\medskip
(ii) The second convergence follows from the first and the following simple observation:
$$N_k\ge L_k;\quad n\ge L_k+2(N_k-L_k)$$
(which must hold as non-empty buses without lonely passengers have at least two passengers),
that is, $$n-(N_k-L_k)\ge N_k,$$
or
$$n/N_k-(N_k-L_k)/N_k\ge 1.$$
To conclude the proof, note that if for some positive integer $r$, $\dfrac{N_{k}-L_{k}}{N_{k}}\ge \dfrac{1}{r}$ (that is, $L_k/N_k\le (r-1)/r$) then 
$N_k\le \dfrac{r}{r+1}n.$  But by (i), the probability of this latter event converges to zero. Hence $L_k/N_k\to 1$ in probability, and so again by (i), $L_k/n\to 1$ in probability too.
\end{proof}

\section{Some results on the support size}\label{sec: support size}
In this  section we analyze the size of the support at a fixed time (and importantly, not “up to that time”).
We recall that $N_k$ denotes the size of the support (i.e. number of non-empty buses) at time $k$. We only give two results here, because the (limiting) distribution of the number of \emph{empty buses} in various regimes has been investigated in \cite{Allocationbook}.
\subsection{The tail of the support size with Stirling numbers} In the sequel the reader may want to consult Appendix A for material on Stirling numbers of the second kind, denoted by $S(a,b).$
\begin{lemma}[Tail of support size]\label{le: explicit} If $n\ge 2$ then for $m\le \min\{n,k\}$,
   $$P^{(n)}(N_k\ge m)=\frac{k!}{(k-m)!}\left(k^{-m}+\sum_{u=m}^{n-1}\frac{1}{k^{u+1}}S(u,m-1)\right).$$
\end{lemma} 
\begin{proof}
For a fixed $k\ge 1,$ we are going to use induction on $n$, by sitting the passengers one by one. Since each passenger chooses a bus uniformly and independently of the others, one has
\begin{align}
P^{(n)}(N_k\ge m)=P^{(n-1)}(N_k\ge m)+P^{(n-1)}(N_k= m-1)\cdot \frac{k-m+1}{k}.
\end{align}
By the definition of $S(a,b)$, we may write
$$P^{(n-1)}(N_k= m-1)=k^{1-n}\binom{k}{m-1}(m-1)!\cdot S(n-1,m-1).$$
(As before, the factor $(m-1)!$ appears because we have labeled buses.)
We thus obtained 
$$P^{(n)}(N_k\ge m)=P^{(n-1)}(N_k\ge m)+k^{-n}\frac{k!}{(k-m)!}\cdot S(n-1,m-1),$$ and by induction, this leads to 
    \begin{align*}P^{(n)}(N_k\ge m)&=P^{(m)}(N_k=m)+\frac{k!}{(k-m)!}\sum_{u=m}^{n-1}\frac{1}{k^{u+1}}S(u,m-1)\\&=k^{-m}\binom{k}{m}m!+\frac{k!}{(k-m)!}\sum_{u=m}^{n-1}\frac{1}{k^{u+1}}S(u,m-1),\end{align*}
as claimed.
\end{proof}
\subsection{Support as birth process for fixed times}
Recall that $N_k$ denotes the support size at time $k$.
\begin{theorem}[Distribution of support size]
For $n$ passengers, the generating function of $N_k$ equals
\[
G_n(z)=\sum_{j=1}^{k}\binom{k-1}{j-1}\left(\frac{j}{k}\right)^{n-1} 
z^{\,j}(1-z)^{\,k-j}.
\]
In particular, 
\begin{align*}
{E}^{(n)}[N_k]
&=G_n'(1)
=k-(k-1)\bigl(1-\tfrac1k\bigr)^{n-1}= k \left[ 1 - \left(1-\frac{1}{k}\right)^n \right]
,
\\
{E}^{(n)}[N_k^2]
&= k^2-(2k-1)(k-1)\left(1-\frac{1}{k}\right)^{n-1}
+(k-1)(k-2)\left(1-\frac{2}{k}\right)^{\,n-1},
\end{align*}
while
\[
\mathsf{Var}^{(n)}(N_k)
= (k-1)\left[\left(1-\frac{1}{k}\right)^{n-1}
- (k-1)\left(1-\frac{1}{k}\right)^{2(n-1)}
+ (k-2)\left(1-\frac{2}{k}\right)^{n-1}\right].
\]
\end{theorem}
\begin{proof}
Fixing $k$ and sitting down the passengers one by one, we see that $(N^{(n)}_k)_{k\ge 1}$ is a pure birth chain on \(\{1,2,\dots,k\}\) with
\begin{align}\label{eq: birthrate}
P^{(n)}(N_{k+1}=m+1\mid N_k=m)=\frac{k-m}{k},\qquad 
P^{(n)}(N_{k+1}=m\mid N_k=m)=\frac{m}{k}.
\end{align}
Starting at \(N_1=1\), for \(n\ge 2\) and \(1\le r\le k\), the distribution is (see \cite{Fellerbook})
\begin{align}\label{eq: distrib}
P^{(n)}(N_k=r)=\binom{k-1}{r-1}\sum_{j=1}^{r}(-1)^{\,r-j}
\binom{r-1}{\,j-1\,}\left(\frac{j}{k}\right)^{n-1}.
\end{align}
(In \cite{Fellerbook} there is a slightly different but equivalent form.)
In fact, the generating function
\(
G_n(z):=\mathbb{E}[z^{X_n}]
\)
admits the closed form 
\[
G_n(z)=\sum_{j=1}^{k}\binom{k-1}{j-1}\left(\frac{j}{k}\right)^{n-1} 
z^{\,j}(1-z)^{\,k-j}.
\]
We include a short proof for completeness, although this result is established (see \cite{Riordan1958}).
Using \eqref{eq: distrib}, (valid for $n\ge2$ and $1\le r\le k$), one has
\[
G_n(z)=\sum_{r=1}^{k}P^{(n)}(X_n=r)\,z^r
=\sum_{r=1}^{k}\binom{k-1}{r-1}z^r
\sum_{j=1}^{r}(-1)^{\,r-j}\binom{r-1}{\,j-1\,}\left(\frac{j}{k}\right)^{n-1}.
\]
Swapping the order of summation (note $j\le r\le k$):
\[
G_n(z)=\sum_{j=1}^{k}\left(\frac{j}{k}\right)^{n-1}
\sum_{r=j}^{k}\binom{k-1}{r-1}\binom{r-1}{\,j-1\,}(-1)^{\,r-j}z^r.
\]
Using the identity
\[
\binom{k-1}{r-1}\binom{r-1}{\,j-1\,}=\binom{k-1}{j-1}\binom{k-j}{\,r-j\,},
\]
writing $r=j+m$ with $m=0,\dots,k-j$:
\begin{align}
G_n(z)
&=\sum_{j=1}^{k}\left(\frac{j}{k}\right)^{n-1}\binom{k-1}{j-1}
\sum_{m=0}^{k-j}\binom{k-j}{m}(-1)^{\,m}z^{\,j+m}\\
&=\sum_{j=1}^{k}\left(\frac{j}{k}\right)^{n-1}\binom{k-1}{j-1}
z^{\,j}\sum_{m=0}^{k-j}\binom{k-j}{m}(-1)^{\,m}z^{\,m},
\end{align}
and finally, recognizing the binomial expansion:
\[
\sum_{m=0}^{k-j}\binom{k-j}{m}(-1)^m z^m=(1-z)^{\,k-j},
\]
it follows that
\[
G_n(z)=\sum_{j=1}^{k}\binom{k-1}{j-1}\left(\frac{j}{k}\right)^{n-1}
z^{\,j}(1-z)^{\,k-j},
\]
as claimed.
\end{proof}
\begin{remark}
From the birth process representation of the support size, it is easy to conclude the first part in Theorem \ref{thm: N.by.Toth}. Indeed, if $n$ is fixed then by \eqref{eq: birthrate}, for $k+1$ buses the birth probabilities are larger than for $k$ buses.
$\hfill\diamond$\end{remark}

\section{Measuring clumping}\label{sec: clumping}
\subsection{Clumping} One possible way to quantify how much ``detached'' a configuration is to measure its ``clumping.''
\begin{definition}
Let $M_k=M_{1,k}$ be the number of passengers in the bus of the passenger with label 1 at time $k$, including that passenger. Let $\tau_1^{\mathsf{lon}}:=\min\{k\ge \lceil n/2\rceil\mid M_{1,k}=1\}$, that is, the first time (not less than $n/2$) when the first passenger becomes lonely.
By Theorem \ref{thm: distr},  $\tau_1^{\mathsf{lon}}<\infty$ a.s.
\end{definition}

The next result says that $M$ behaves as a supermartingale between $\lceil n/2\rceil$ and $\tau_1^{\mathsf{lon}}$.
\begin{lemma}\label{lem: fellow}
$\{M(k\wedge \tau_1^{\mathsf{lon}})\}_{k\ge \lceil n/2\rceil}$ is a supermartingale. 
\end{lemma}
\begin{proof}
The passenger we follow either relocates with probability $1/k$ or stays with probability 
$(k-1)/k$, hence, when $n\le 2k$ and $a\ge 2$,
\begin{align*}&E^{(n)}(M_k\mid M_{k-1}=a, M_{k-2}=b_{k-2},...,M_{1}=b_1)=E^{(n)}(M_k\mid M_{k-1}=a)\\
&=(1/k)\mathsf{E} [1+\mathsf{Bin}(n-1,1/k)]+(k-1/k)\mathsf{E}[(a-\mathsf{Bin}(a-1,1/k)]
\\
&=(1/k) [1+(n-1)/k]+(k-1/k)[(a-(a-1)/k]\le a,
\end{align*}
since the inequality is equivalent to $n-1\le (2k-1)(a-1).$
(We do not have to consider $a=1$ because of the stopping.)
\end{proof}
\begin{definition}[Measuring clumping]\label{def: clumping.measure}
Let $W^i_k$ be the number of passengers in bus number $i$ for $1\le i\le k$ at time $k$.  Let $\widehat M_k:=\sum_{1\le i\le k} (W^i_k)^2$.
\end{definition}
Note that $\widehat M_k$ measures how ``clumped'' the configuration is. Indeed if in a given bus, the passengers are re-distributed into a number ($>1$) of empty buses then $\widehat M_k$ decreases. Hence $n\le \widehat M_k\le n^2$. Furthermore, since $\sum_{1\le i\le k} W^i_k=n$, by the Cauchy-Schwarz inequality, $$\widehat M_k\ge \frac{n^2}{k},$$ and equality corresponds to the case when $k$ divides $n$ and each bus has $n/k$ passengers. Thus, altogether, one has
$$n^2\ge\widehat M_k\ge \max\left\{n, \frac{n^2}{k}\right\}.$$
Again, by Theorem \ref{thm: distr}, with probability one, there exists a $k_0$ such that $\widehat M_k=n, k\ge k_0$.

Let $$\tau^{\mathsf{lon
}}:=\tau_1^{\mathsf{lon
}}\wedge \dots \wedge \tau_n^{\mathsf{lon
}}$$ be the first time $k\ge n/2$ that there exists at least one lonely passenger, where $\tau_j^{\mathsf{lon
}}$ is defined similarly as before.
\begin{corollary}\label{cor: clump}
 The process $\{\widehat M(k\wedge \tau^{\mathsf{lon
}})\}_{k\ge \lceil n/2\rceil}$ is a supermartingale.
\end{corollary}
\begin{proof}
If $M_{i,k}$ is defined similarly to $M_{1,k}$, then we have $\widehat M_k=\sum_{i=1}^n M_{i,k}$ (since in a given bus, the number of travelers are counted as many times as that number is) and $M(i\wedge \tau_i^{\mathsf{lon}})$ is a supermartingale starting with $\lceil n/2\rceil$ by Lemma \ref{lem: fellow}, hence so is
$M(i\wedge \tau_i^{\mathsf{lon}}\wedge \tau^{\mathsf{lon}})=M(i\wedge  \tau^{\mathsf{lon}}).$ This is then true for their sums as well.
\end{proof}
\subsection{Relative clumping}
\begin{definition}
We define the relative clumping (that is, relative to the number of passengers) by
$$\mathsf{RC}_{n,k}:=\log (n^{-1}\widehat M_k)$$ and note that $\mathsf{RC}_{n,k}\in [0,\log n]$ for all $k\ge 1$ Furthermore, $\mathsf{RC}_{n,k}=0$ if and only if $n\le k$ and the $n$ balls are fully detached at $k$. On the other hand $\log n$ is only attained when every passenger is in the same bus at $k$.
\end{definition}

Next, we show that the tail of $\mathsf{RC}_{n,cn}$ is (at least) exponentially small, uniformly in $n$.
\begin{proposition}[Sharp drop of rel. clumping in linear time]
 If $c>0$ then $$\max_{n\ge 1}P^{(n)}(\mathsf{RC}_{n,\lfloor cn\rfloor}>x)\le (1+c^{-1})e^{-x},\quad x>0.$$
\end{proposition}
\begin{proof}
    Since $W_k^i$ is $\mathsf{Bin}(n,1/k)$-distributed, $E^{(n)}(W_k^i)^2=\dfrac{n}{k}\left(\dfrac{n-1}{k}+1\right)$. The statement then follows from the Markov-inequality.
\end{proof}
The proposition demonstrates that while complete “decumpling” (i.e., complete detachment) requires roughly a quadratic time interval, relative clumping significantly decreases from the initial $\log n$ below a constant (dependent on $c$) already within $cn$ time. To illustrate this phenomenon, we present Fig. \ref{fig:clumping_comparison}, which depicts a pronounced drop in relative clumping at the outset and, respectively, zoom in to reveal fluctuations within the initial 500 steps. These simulations were conducted using Chat GPT 5.

\begin{remark} Since 
$P(\text{detachm. in}\ [n/2,k(n)])=P\left(\min_{i\in [n/2,k(n)]}\widehat M(i\wedge \tau^{\mathsf{lon}}\right)$, it is tempting to apply Doob's inequality
to the non-negative submartingale 
$\mathcal{M}_i:=n^2-\widehat M(i\wedge \tau^{\mathsf{lon
}})$ on the time interval $[n/2,k(n)]$, but unfortunately, that does not seem to yield a tight estimate.
\end{remark}

\medskip
\begin{figure}[h!]
    \centering
    \begin{subfigure}[t]{0.48\textwidth}
        \centering
        \includegraphics[width=\linewidth]{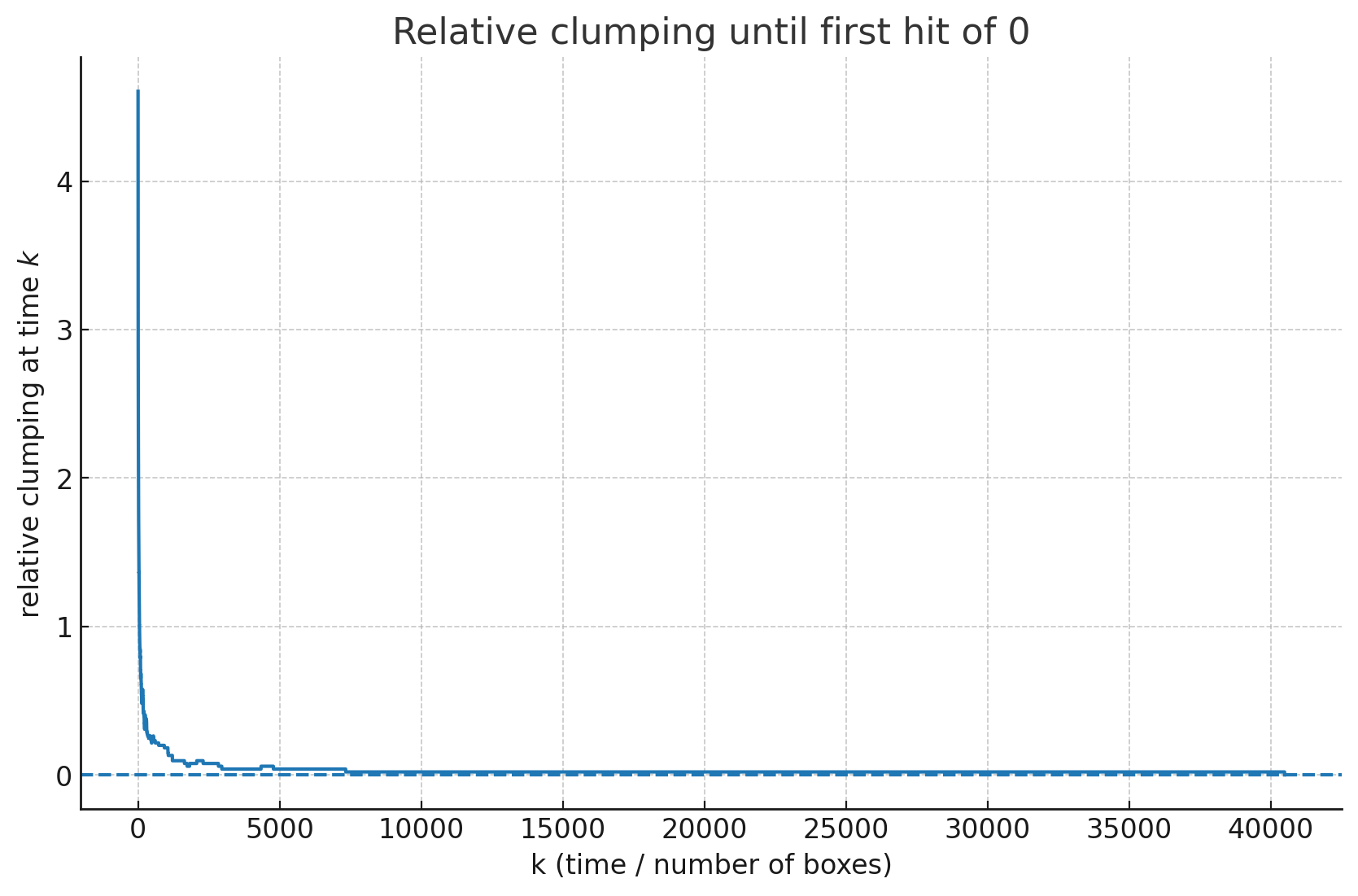}
        \caption{De-clumping for $n=100$ balls -- the vertical axis is the relative clumping value at the given time. Full detachment at time $k=40490$.}
        \label{fig:clump}
    \end{subfigure}
    \hfill
    \begin{subfigure}[t]{0.48\textwidth}
        \centering
        \includegraphics[width=\linewidth]{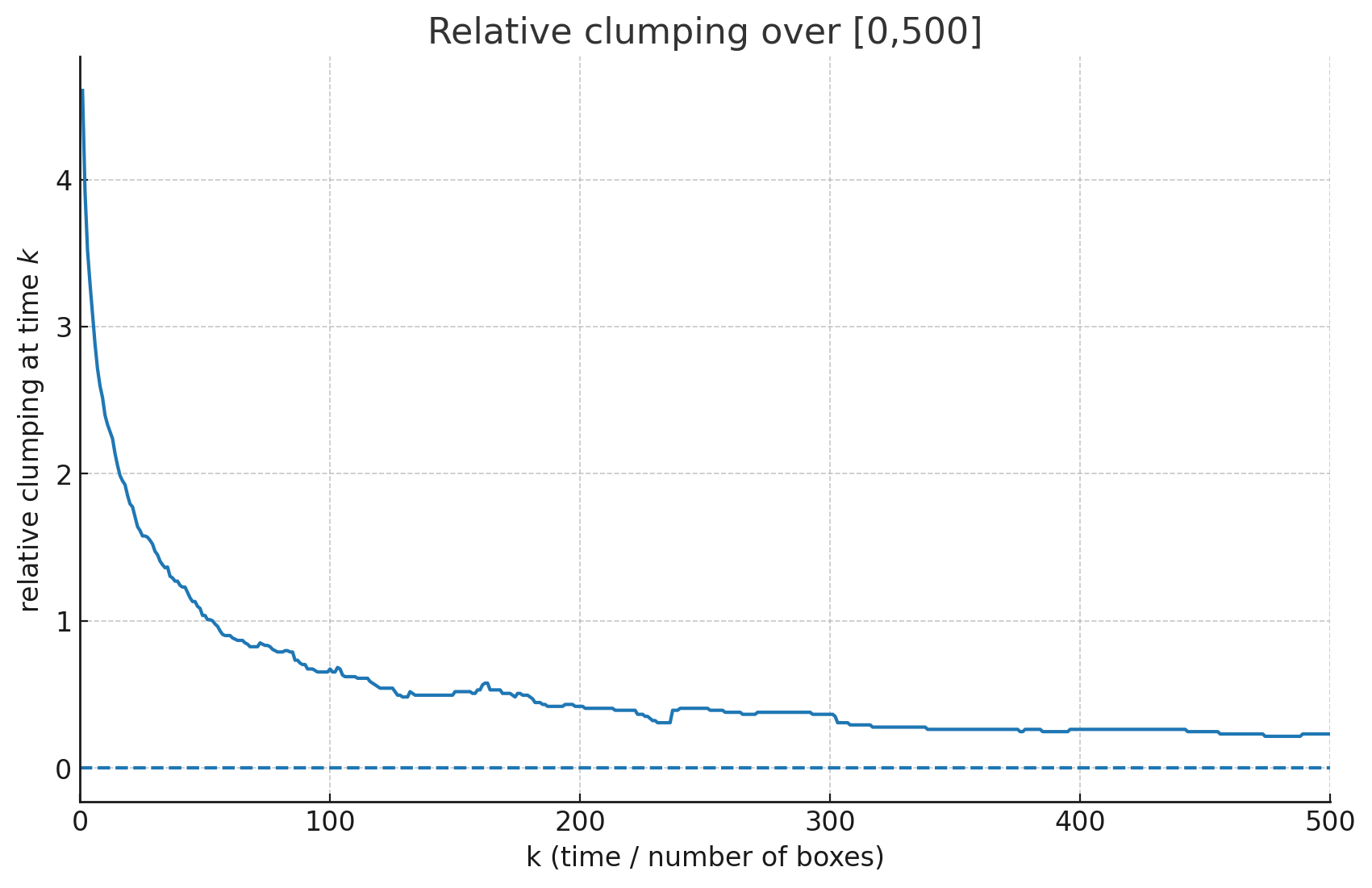}
        \caption{Zooming in -- fluctuation until time $500$.}
        \label{fig:clump_zoom}
    \end{subfigure}
    \caption{De-clumping dynamics for $n=100$ balls; large drop already in linear time!}
    \label{fig:clumping_comparison}
\end{figure}
\section{Poisson number of passengers; a result on comparing binomials}\label{sec: Poissonization}
When the number of passengers $\mathcal{N}$ is not a fixed number $n$, but is random and Poisson distributed with parameter $\lambda>0$, the calculations and proofs are easier. In this case, the number of passengers in any given bus is obtained by ``thinning'' (or splitting) the Poisson distribution and thus it is itself Poisson distributed with parameter $\lambda/k$, and those passenger numbers are independent. For example,
$$P_{\lambda}(\text{full detachment at\ } k)=P_{\lambda}(L_k=\mathcal{N})=\left[e^{-\lambda/k}(1+\lambda/k)\right]^k=e^{-\lambda}(1+\lambda/k)^k,$$
where $P_{\lambda}$ refers to the law when the passenger number has parameter $\lambda.$
In particular, $$\lim_{k\to\infty} P_{\lambda}(\text{full detachment at\ } k)=1,$$ and a bit more generally,  if $\lambda(k)=o(k)$, then $\lim_{k\to\infty} P_{\lambda}(\text{full detachment at\ } k)=1$ still holds. However, for any $\lambda=\lambda(k):=ck,c>0$,
$$\lim_{k\to\infty} P_{\lambda}(\text{full detachment at\ } k)=0,$$ exponentially fast. 

We next note that, in the case of a Poissonian passenger number (with a fixed $\lambda$) it is very easy to prove the monotonicity of the probability in Theorem \ref{thm: Toth}, without the sophistication of the proofs in \cite{Haslegrave, Toth}. 
Indeed, using the independence of passenger numbers in the buses and the Poissonian thinning explained above, the probability of not having any bus with a single  passenger is
$$(1-(\lambda/k) e^{-\lambda/k})^k.$$
Hence, it is enough to check that the function $f(x):=(1-xe^{-x})^{1/x}$ is nonincreasing on the positive axis. Now, for $x>0,$
$\mathsf{sgn} (f'(x))=-\mathsf{sgn}(\log  x(x-1))\le 0$, and we are done.

For the more general statement on stochastic dominance, let us first introduce the notation $\prec,\preceq$ for (strict) stochastic dominance. 

\subsection{Comparing binomial distributions} We will also need the following (unpublished) theorem due to J. Najnudel \cite{Najnudel}, which is of independent interest.
\begin{lemma}[Najnudel's theorem]\label{thm: Najnudel}
   Let $X\sim\mathsf{Bin}(n,p)$ and $Y\sim\mathsf{Bin}(m,q)$. For $Y \preceq X$ to hold, it is necessary and sufficient that $n\ge m$ and \begin{align}\label{eq: enough}\mathbf{P}^{n,p}(X=0)=(1-p)^n\le(1-q)^m=\mathbf{P}^{m,q}(Y=0).
   \end{align}
\end{lemma}
Note that necessity is trivial. 
Since Najnudel's proof is unpublished, we are going to provide here a relatively simple proof (different from Najnudel's).
\begin{proof} First, $n<m$ rules out the domination, as $X=m$ is impossible. When $n=m$, domination occurs if and only if $p\ge q$
(by representing $X,Y$ as  sums of i.i.d. Bernoulli variables). Hence, it is enough to assume that $m<n$ and check the sufficiency of \eqref{eq: enough}.

Let $\mathsf{PPP}(\lambda)$ and $\mathsf{PPP}(\widehat\lambda)$ denote the Poisson point process on $[0,n)$ with constant intensity $\lambda>0$ and the Poisson point process won $[0,m)$ with constant intensity $\widehat \lambda>0$, respectively. By ``cell $i$'' we will mean the interval $[i-1,i),i\ge 1$. These point processes induce binomial distributions with parameters $n,p$ and $m,q$ respectively if the ``$i$th success'' means that cell $i$ is non-empty, and if $1-p=e^{-\lambda}$ and $1-q=e^{-\widehat\lambda}$. Thus, the assumption in the lemma is tantamount to $e^{-n\lambda}\le e^{-m\widehat\lambda}$,
that is $n\lambda\ge m\widehat\lambda$. We need to show that the last inequality implies the stochastic dominance in the statement. 

Let us re-partition $[0,n)$ into $m$ equal intervals, each of length $n/m$ (that is, into the sub-intervals $[0,n/m),[n/m,2n/m),...,[n-n/m,n)$, which yields cell-size $n/m$). This new partition with the same $\lambda$ induces a different binomial distribution than before, namely the one with parameters $n$ and $1-e^{-(n/m)\lambda}$. Since $n\lambda\ge m\widehat\lambda$, this last binomial distribution stochastically dominates the one with parameters $(m,q)=(m,1-e^{-\widehat\lambda})$. Therefore, to complete the proof, it is enough to show that with $\lambda$ fixed, the number of non-empty  cells in the $n$-partition of $[0,n)$ stochastically dominates the number of non-empty  cells in the $m$-partition of $[0,n)$.
This is accomplished by noting that $\mathsf{PPP}(\lambda)$ on $[0,n)$ can be realized by picking a Poisson($\lambda n$) number of random points on $[0,n)$, one after the other, at each unit time, independently, such that each point's location is uniformly distributed on $[0,n)$. The number of non-empty cells is going to be non-decreasing during this process up to time $T$ which is Poisson($\lambda n$)-distributed. That is, we have a pure birth process observed at a random time, and at that time we are interested in the number of non-empty cells. Finally, the statement follows from the fact that one birth rate dominates the other, and it is well known that this implies that the first birth process stochastically dominates the other at \emph{each} given time. Indeed, if the current state of the birth process is $y$ then the probability of birth is $1-y/n$ resp. $1-y/m$. The proof is complete.
\end{proof}

\subsection{T\'oth's theorem in the Poissonian case}
With Najnudel's theorem in our toolset, we can now verify T\'oth's theorem in the Poissonian setup.
\begin{theorem}[T\'oth's theorem in the Poissonian case] Let $\lambda>0$ and $k_2>k_1.$ Then
\begin{equation}\label{larger}
    L_{k_1} \preceq L_{k_2}.
\end{equation}
\end{theorem}

\begin{proof} In fact even $\prec$ holds. By independence,
$L_{k_i}$ has distribution $\mathsf{Bin}(k_i,\frac{\lambda}{k_{i}}e^{-\lambda/k_i}),$ because in one cell the probability of receiving a single point is $\frac{\lambda}{k_{i}}\exp(-\lambda/k_i)$. Since $k_1<k_2$, we may use Lemma \ref{thm: Najnudel}, for the stochastic dominance between binomials with
$m=k_1,n=k_2$, $q=(\lambda/k_1) \exp(-\lambda/k_1)$ and $p=(\lambda/k_2) \exp(-\lambda/k_2)$. For the second condition to hold, introduce $u_i:=\lambda/k_i$, and the condition becomes that
$f(u_2)\ge f(u_1),$ where $f(x):=(1-xe^{-x})^{1/x}$ and as we have seen, $f$ is non-increasing, and strictly decreasing for $0<x\neq 1$, finishing the proof.
\end{proof}

\section{Open problems}\label{sec: open pr}
Below, we present several open problems in the hope of raising interest in the model.
\begin{itemize}
    
\item ({\bf Problem 1}: Stopping times) 
    
(a) The  analysis throughout this paper establishes that the asymptotic scaling laws relevant to our model are heterogeneous: the correct scaling factor depends on the question, sometimes involving a logarithmic correction and sometimes not. These findings do not, by themselves, characterize the critical growth rate of $k(n)$ necessary to describe the behavior of the tail probability $P(\hat \tau^{(n)}\le k_n)$. In light of Theorem \ref{thm: quadratic}, it would be desirable to determine whether the relevant scaling  is closer to $n^2/\log n$ or $n^2$. Numerical simulations, however, favor purely quadratic scaling; see Fig.\ref{fig:figure1} and Fig.\ref{fig:figure2}.

(b) What is the law of the first time when at least one lonely passenger exists? (This is slightly different from $\tau^{\textsf{lon}}$ defined before Corollary \ref{cor: clump}.)

\begin{figure}[h!] 
    \centering
    \begin{minipage}[t]{0.38\textwidth} 
        \centering
        \includegraphics[width=\linewidth]{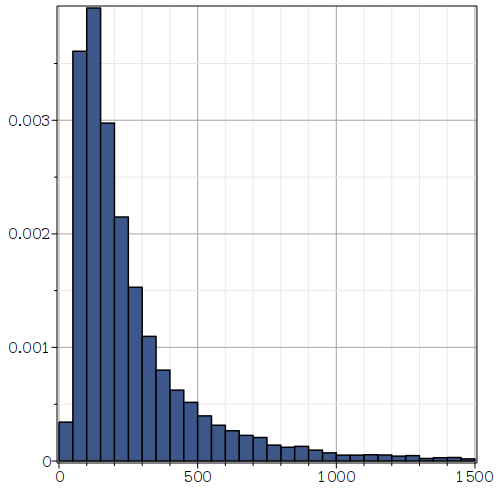}
        \caption{Simulated distribution bar graph for the first detachment time $\hat\tau^{(20)}$. The mean is $\approx 322$. [Courtesy of S. Volkov.]}
        \label{fig:figure1}
    \end{minipage}
    \hfill 
    \begin{minipage}[t]{0.38\textwidth} 
  \centering
\includegraphics[width=\linewidth]{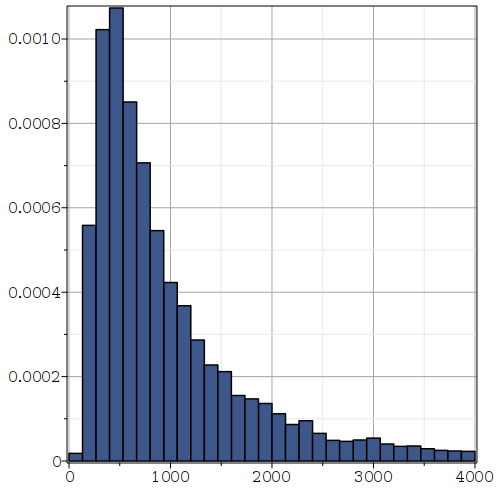}
        \caption{Simulated distribution bar graph for the first detachment time $\hat\tau^{(40)}$. The mean is $\approx 1270$. [Courtesy of S. Volkov.]}
        \label{fig:figure2}
    \end{minipage}
    \label{fig:combined_figures}
\end{figure}

\item ({\bf Problem 2}: Point processes)  We consider three point processes on $\mathbb{Z}_+$:

(a) For a fixed $n$, what is the distribution of the point process formed by {\it all} detachment times? We know that there are finitely many of them a.s., so the first question is the distribution of their number. 

(b) Similarly, consider the (larger) point process formed by all times $k$ when the process is in a state of detachment ($L_k=n$). By Proposition \ref{prop: expect} 
we know how the \emph{intensity measure} of the latter point process behaves for \emph{large $n$} and we also have a scaling limit result in Theorem \ref{thm: scaling.limit.state.det}, but we know very little about the point processes for a finite $n$.

(c) Let $n\ge 2$ and $$\tau^{(m,n)}:=\min\{k\ge 1\mid L_{k+j}\ge m,\ \forall j\ge 0\}.$$ Then $\tau^{(m,n)}<\infty$ a.s., since we know that $\tau^{(m,n)}\le \tau^{(n,n)}=\tau^{(n)}<\infty$ a.s. What can we say about the (finite) point process $\{\tau^{(m,n)},\ m=1,2,...,n$\} for some fixed $\ge 2$?

\item ({\bf Problem 3}: Clumping) 

(a) Fix $n\ge 1$. What are some interesting properties of the relative clumping {\it process} $(\mathsf{RC}_{n,k})_{k\ge 1}$? 

(b) Is $\widehat M_{k+1}$ (defined in Definition \ref{def: clumping.measure}) stochastically dominated by $\widehat M_{k}?$ That is, is it true the ``clumping is stochastically monotone non-increasing'' in $k$? In general, analyzing the process $\widehat M$ is more difficult then the support process $N$, which is a pure birth process.

 \item ({\bf Problem 4}: Lonely walkers) This problem is not about the detachment process but it is inspired by it. 
 
In the detachment process we have $n$ independent copies of the same random walk. The random walk itself is very simple but time inhomogeneous:
at time $k$ the particle jumps to position $k$ with probability $1/k$, irrespective where the particle is at time $k-1$; otherwise the particle stays.

    This inspires the following question: consider $n$ independent copies of a unit-time simple random walk on $\mathbb Z^d, d\ge 1$ starting at the origin and making $t$ steps. A walker is ``lonely'' at time $t$ if no other walker shares its location at that time. Is it true that the probability that at least one lonely particle exists is monotone increasing in $t$ for $n$ fixed?

     As S. Volkov pointed out, in dimension one with $n=2$ the statement is true. This is because the difference between two independent walkers is the same as a single walker considered in every even step, and monotonicity follows from the easy to check fact that
     $\binom {2t}{t} /2^{2t}$ is (strictly) monotone in $t$.
\end{itemize}
We conclude with some conjectures.

\begin{conjecture}[Scaling limit for the first detachment time]
We conjecture that $\hat \tau^{(n)}/n^2$ has a limiting distribution as $n\to\infty$. 
\end{conjecture}

\begin{conjecture}[Diffusive scaling for processes]
Consider the detachment process which couples all $n$-detachment processes as in subsection \ref{subsect: couple}. Then $\hat \tau^{(n)}, n=2,3,...$ are defined on the same probability space. For $k\ge 1$ one can define $\mathcal{M}(k):=\max\{n:\ \hat\tau^{(n)}\le k\}$, that is $\mathcal{M}(k)$ is the largest number $n$ for which the first $n$ passengers have already been separated at least once by time k. ($\mathcal{M}(k):=1$ when no two passengers have been separated.) We conjecture that for the process $\mathcal{M}(k),k\ge 1$ (which  clearly has monotone non-decreasing paths, and can be extended to all  times  $t\ge 1$ by linear interpolation) diffusive scaling applies in the sense that the laws of the processes $\mathcal{M}^{(m)}(t):=\mathcal{M}(m^2 t)/m$ have a limit as $m\to\infty$.
\end{conjecture}

\section{Appendix A: Stirling numbers of the second kind}\label{sec:appendix} In the paper, the following notion from combinatorics is used: A {\it Stirling number of the second kind} (or Stirling partition number) is the number of ways to partition a set of $a$ objects into $b\le a$ non-empty subsets and is denoted by
$S(a,b)$.
Here one considers {\it labeled} objects and {\it unlabeled} subsets.
Some useful known results  on Stirling numbers of the second kind are collected below. 

The following asymptotics is known \cite{Temme}  when $n/m=v>1$ and $n\to \infty$:
\begin{align}\label{eq: Temme}
    S(n,m)\sim \alpha_v\cdot \beta_v^{n}\frac{m^{n}}{n^{m}}\gamma_v^m\binom{n}{m}, 
\end{align}
where $\alpha_v,\beta_v,\gamma_v$ are explicit positive constants, depending only on $v$. For example, if $v=2$, then using that
$$\binom{2m}{m}\sim \frac{2^{2m}}{\sqrt{m\pi}},\ \text{as}\ m\to\infty,$$ one obtains that
$$S(2m,m)\sim k_v^m \cdot m^m (m\pi)^{-1/2},$$ with some $k_v>0.$ Similar asymptotics can be obtained for $\binom{\lfloor cm\rfloor}{m}$
for any $c>1$.

In \cite{RennieDobson} one finds the bounds:  $$\frac{1}{2}(b^2+b+2)b^{a-b-1}-1\le S(a,b)\le (1/2)\binom{a}{b}b^{a-b},$$
while a handy summation formula is
\begin{align}\label{eq: summ.formula}
   S(a+1,b+1)=\sum_{j=b}^{a}(b+1)^{a-j}S(j,b). 
\end{align}
For further reading on Stirling numbers of the second kind, we recommend \cite{RennieDobson}.

\section{Appendix B: Tightness and  \text{càdlàg} version for binary  processes}\label{sec: tight}

{\bf Tightness:} Consider $D([0,1]),$ the Skorokhod space of \text{càdlàg} functions on the unit interval, equipped with the Skorokhod topology (see Chapter 3 in \cite{Billingsley1999}).
Theorem 13.2 (p. 139) in \cite{Billingsley1999}  gives a necessary and sufficient condition for tightness in $D([0,1]);$ one can of course apply it to $[a,b]$ in place of $[0,1]$. In the particular case when the  $(X^{(n)},P^{(n)})$ are $0-1$-valued c\`adl\`ag processes, it takes the following simpler form.

For a given $\delta>0$, a decomposition of the interval $[0,m)$ into the disjoint sub-intervals 
$[t_{i-1},t_i)_{i=1,...,v}$ is called $\delta$-sparse if $t_i-t_{i-1}>\delta$ for each $i=1,...,v.$
Let $A^{(n)}(m,\delta)$ be the event that that there exists a $\delta$-sparse decomposition of $[0,m)$ such that $X^{(n)}$ is constant on each sub-interval $[t_{i-1},t_i).$ Let us call the trajectory in this case ``$\delta$-flat on $[0,m).$''
Then the $0-1$-valued c\`adl\`ag processes $(X^{(n)},P^{(n)})$ on $[0,\infty)$  are tight if and only if for every $m>0$,
\begin{align}\label{eq: Billingsley.tightness}
\lim_{\delta\downarrow 0}\liminf_{n\to\infty}P^{(n)}(A^{(n)}(m,\delta))=\lim_{\delta\downarrow 0}\liminf_{n\to\infty}P^{(n)}(X^{(n)}\ \text{is}\ \delta-\text{flat on\ }[0,m))=1.
\end{align}
Intuitively, the criterion requires that the paths \emph{do not fluctuate too much} between zero and one. (Note that if $f:[0,\infty)\to\{0,1\}$ is càdlàg  then $[0,\infty)$ can be written as a disjoint union of half-open intervals of positive length, on each of which $f$ is constant.)

A condition due to Chentsov guaranteeing tightness in $D([0,T])$ is as follows (see e.g. \cite{Kunita1986}): there exist positive constants, $K$, $\gamma$ and $\alpha$, not depending on $n$, such that
$$E^{(n)}(|X^{(n)}_{t}-X^{(n)}_{t_1}|^{\gamma}\,|X^{(n)}_{t_2}-X^{(n)}_t|^{\gamma})\le K|t_2-t_1|^{1+\alpha},\ 0\le t_1<t_2\le T,$$
and
$$E^{(n)}|X^{(n)}_t|^{\gamma}\le K,\ 0\le t\le T.$$
For $0-1$-valued processes, only the first inequality is relevant and it simply becomes
$$E^{(n)}|(X^{(n)}_{t}-X^{(n)}_{t_1})(X^{(n)}_{t_2}-X^{(n)}_t)|\le K|t_2-t_1|^{1+\alpha},\ 0\le t_1<t_2\le T,$$
that is, one needs that there exist positive constants, $K$ and $\alpha$, not depending on $n$, such that
\begin{align}\label{eq:Chentsov.for.01}
P^{(n)}(X^{(n)}_{t}\neq X^{(n)}_{t_1},X^{(n)}_{t}\neq X^{(n)}_{t_2})\le K|t_2-t_1|^{1+\alpha},\ 0\le t_1<t_2\le T,
\end{align}

\medskip
{\bf Existence of \text{càdlàg} version:} Theorem 13.6 (p. 143) in \cite{Billingsley1999} gives a sufficient condition for the existence of a version of a stochastic process with values in $D([0,T])$. In case of a 0-1-valued process it implies that the following  condition is sufficient: 
\begin{align}\label{eq:cadlag.for.01}
\lim_{h\downarrow 0}P(X_{t+h}\neq X_t)&=0;\nonumber\\
P(X_{t}\neq X_{t_1},X_{t}\neq X_{t_2})&\le K|t_2-t_1|^{1+\alpha},\ 0\le t_1<t_2\le T.
\end{align}

\section{Acknowledgments}
I am grateful to M. Bal\'azs (University of Bristol) for initially pointing out that the analysis of T\'oth's original problem simplifies considerably when considering a Poisson number of passengers, as discussed in Section \ref{sec: Poissonization}. I owe thanks to J. Najnudel (University of Bristol), who originally answered my question about binomials (see Lemma \ref{thm: Najnudel}) in \cite{Najnudel}. I thank S. Volkov (Lund University) for kindly providing several useful simulations, particularly for the first detachment times. Finally, I am deeply indebted to I. P. T\'oth for inspiring this work and for making me a firm admirer of his ``lonely passenger problem.''

\newpage

\end{document}